\documentclass[a4paper,11pt]{article}
\usepackage{amsmath,amsthm,amssymb,enumitem,xcolor}
\usepackage{tikz}

\usepackage[hang]{footmisc}
\usepackage[nosort,nocompress,noadjust]{cite}

\usepackage[bookmarks=false,hyperfootnotes=false,colorlinks,
    linkcolor={red!60!black},
    citecolor={blue!50!black},
    urlcolor={blue!80!black}]{hyperref}

\renewcommand{\eqref}[1]{\hyperref[#1]{(\ref{#1})}}

\usepackage{faktor}
\usepackage{youngtab}

\newlist{enumlist}{enumerate}{1}
\setlist[enumlist]{labelindent=0cm,label=\arabic*.,labelwidth=2.5ex,labelsep=0.5ex,leftmargin=3ex,align=left,topsep=0.5ex,itemsep=1ex,parsep=1ex}

\newlist{itemlist}{itemize}{1}
\setlist[itemlist]{labelindent=0cm,label=$\bullet$,labelwidth=2.5ex,labelsep=0.5ex,leftmargin=3ex,align=left,topsep=0.5ex,itemsep=1ex,parsep=1ex}

\numberwithin{equation}{section}

{\theoremstyle{definition}\newtheorem{definition}{Definition}[section]
\newtheorem*{definition*}{Definition}

\newtheorem{remark}[definition]{Remark}

\newtheorem*{example*}{Example}
\newtheorem*{examples*}{Examples}
\newtheorem*{oproblem}{Open Problem}}

\newtheorem{proposition}[definition]{Proposition}
\newtheorem{lemma}[definition]{Lemma}
\newtheorem{theorem}[definition]{Theorem}
\newtheorem{corollary}[definition]{Corollary}

\newtheorem{conjecture}[definition]{Conjecture}
\newtheorem{letterthm}{Theorem}

{\theoremstyle{definition}}

\newcommand{\C}{\mathbb{C}}
\newcommand{\cC}{\mathcal{C}}

\newcommand{\ot}{\otimes}

\newcommand{\Z}{\mathbb{Z}}

\newcommand{\cO}{\mathcal{O}}

\newcommand{\N}{\mathbb{N}}

\newcommand{\cG}{\mathcal{G}}

\newcommand{\cF}{\mathcal{F}}

\newcommand{\cM}{\mathcal{M}}

\newcommand{\cP}{\mathcal{P}}

\newcommand{\bP}{\mathbb{P}}

\newcommand{\Y}{\mathbb{Y}}
\newcommand{\G}{\mathbb{G}}

\DeclareMathOperator{\End}{End}

\DeclareMathOperator{\mult}{mult}

\DeclareMathOperator{\tr}{tr}

\DeclareMathOperator{\Br}{Br}
\DeclareMathOperator{\Part}{Part}

\DeclareMathOperator{\Ind}{Ind}

\begin{document}

\begin{center}
{\boldmath\Large\bf Traces on diagram algebras II: Centralizer algebras of easy groups and new variations of the Young graph}

\bigskip

{\sc by Jonas Wahl\footnote{\noindent Hausdorff Center for Mathematics, Bonn (Germany).\\ E-mail: wahl@iam.uni-bonn.de.}}

\end{center}

\begin{abstract}
\noindent In continuation of our recent work \cite{Wa20}, we classify the extremal traces on infinite diagram algebras that appear in the context of Schur-Weyl duality for Banica and Speicher's easy groups \cite{BS09}. We show that the branching graphs of these algebras describe walks on new variations of the Young graph which describe curious ways of growing Young diagrams. As a consequence, we prove that the extremal traces on generic rook-Brauer algebras are always extensions of extremal traces on the  group algebra $\C[S_{\infty}]$ of the infinite symmetric group. Moreover, we conjecture that the same is true for generic parameter deformations of the centralizers of the hyperoctahedral group and we reduce this conjecture to a conceptually much simpler numerical statement. Lastly, we address the trace classification problem for the Schur-Weyl dual of the halfliberated orthogonal group $O_N^*$, in which case extremal traces are always extensions of extremal traces on $\C[S_{\infty} \times S_{\infty}]$. Our approach relies on methods developed by Vershik and Nikitin in \cite{VN06}.
\end{abstract}

\section{Introduction}

The question of how to classify extremal traces on inductive limits of finite-dimensional semisimple algebras lies at the very heart of asymptotic representation theory and has led to a bouquet of beautiful results as well as a large number of applications in other areas of mathematics (see e.g. \cite{VK82} \cite{Bn98}\cite{O03} \cite{BO16} \cite{LPW19} just to name a few). In this article, we address this question for \emph{diagram algebras}, a class of algebras whose origins can be traced back to the works of Brauer \cite{Br37} and Weyl \cite{W46} on what is known today as Schur-Weyl duality for compact groups. Since these early works, a myriad of interesting diagram algebras have been found in the context of statistical mechanics \cite{MaS94} \cite{Ma96} \cite{Ma00} \cite{Jo94}, subfactor and knot theory \cite{Jo83} \cite{BiJo95} and quantum groups \cite{BS09} \cite{We13} \cite{RaWe16} and the literature discussing their properties and applications is vast and can not be done justice here (for a start, see \cite{HR05} \cite{GL04} \cite{COSSZ20} and the references therein). Although diagram algebras admit a natural infinite-dimensional direct limit object, until recently the trace classification question had only been addressed for few examples in the literature \cite{Was81} \cite{VN06}, see also \cite{VN11}. In our recent work \cite{Wa20}, we therefore started to present a first set of answers to the trace classification question by classifying extremal traces for a first class of diagram algebras called \emph{noncrossing} or \emph{planar} diagram algebras. For such an algebra, the trace classification problem could be solved by converting it into equivalent more approachable classification problem for random lattice paths or random walks on trees (depending on the specific algebra), see \cite{Wa20}.

In this article, we classify extremal traces for diagram algebras containing the so-called \emph{simple crossing}. Thanks to the classification of \emph{categories of set partitions} with the simple crossing in \cite{BS09} yielding exactly six examples (denoted by $\cC_{\mathcal{S}}, \cC_{\mathcal{S}'},\cC_{\mathcal{H}},\cC_{\mathcal{B}}, \cC_{\mathcal{B}'} ,\cC_{\mathcal{O}}$), one can deduce that there are four inductive series of diagram algebras of interest (since $\cC_{\mathcal{S}}$ and $\cC_{\mathcal{S}'}$ give the same series and the same is true for $\cC_{\mathcal{B}}$ and $\cC_{\mathcal{B}'}$). These series are
\begin{itemize}
\item The \emph{partition algebras} $ \mathrm{P}_{\delta}(k)= A_{(\mathcal{S},\delta)}(k) , \ k \geq 0$ discovered by Martin  \cite{MaS94} \cite{Ma96} \cite{Ma00} and Jones \cite{Jo94};
\item The \emph{Brauer algebras} $ \Br_{\delta}(k) = A_{(\mathcal{O},\delta)}(k), \ k \geq 0$ \cite{Br37} that were thoroughly analyzed by Wenzl in \cite{Wen88};
\item The \emph{rook-Brauer algebras} $\mathrm{rBr}_{\delta}(k)= A_{(\mathcal{B},\delta)}(k) , \ k \geq 0$  studied in \cite{dMH13};
\item the algebras $A_{(\mathcal{H},\delta)}(k), \ k \geq 0$ studied in \cite{Or05} which do not seem to have a specific name in the literature.
\end{itemize}
All of these algebras depend on an additional complex parameter $\delta \in \C$, called the \emph{loop parameter}. When this parameter is chosen to be a positive integer $\delta =n \in \N$, and when $k$ is small enough compared to $n$ (for instance $2k < n$ for the partition algebra), then all of these algebras are isomorphic to the centralizer algebra of a compact group in a tensor power of their standard representation:
\begin{itemize}
\item the partition algebra $ P_{n}(k)$ is isomorphic to the centralizer algebra $\End_{S_n}(V^{\ot k})$ of the symmetric group $S_n$;
\item the Brauer algebra $ B_{n}(k)$ is isomorphic to the centralizer algebra $\End_{O_n}(V^{\ot k})$ of the orthogonal group $O_n$;
\item the rook-Brauer algebra $ rB_{n}(k)$ is isomorphic to the centralizer algebra $\End_{B_n}(V^{\ot k})$ of the bistochastic group $B_n$;
\item the  algebra $ A_{(\mathcal{H},n)}(k)$ is isomorphic to the centralizer algebra $\End_{H_n}(V^{\ot k})$ of the hyperoctahedral group $H_n$.
\end{itemize} 

Unfortunately, for the choice $\delta =n$, none of these algebras are semisimple when $2k \geq n$, see e.g. \cite{FM20}. However, whenever one chooses $\delta$ 'generically' (see Remark \ref{rem.genericparameter}), the diagram algebras above are semisimple for all $k \geq 0$ and therefore admit infinite versions, that is to say direct limit algebras $A_{(\mathcal{X},\delta)}(\infty), \ \mathcal{X} = \mathcal{S}, \mathcal{O}, \mathcal{B}, \mathcal{H}$.

Most of the representation theoretic data of the algebra $A_{(\mathcal{X},\delta)}(\infty), \ \mathcal{X} = \mathcal{S}, \mathcal{O}, \mathcal{B}, \mathcal{H}$ can be encoded in the \emph{branching graph} or \emph{Bratelli diagram} of the sequence $A_{(\mathcal{X},\delta)}(1) \subset A_{(\mathcal{X},\delta)}(2) \subset \dots$. In short, the vertices on level $n$ of this graded bipartite graph are the irreducible representations of $A_{(\mathcal{X},\delta)}(n)$ and the number of edges from an $n-1$-th level representation $\pi$ to the $n$-th level representation $\rho$ is the multiplicity of $\pi$ in the restriction of $\rho$ to $A_{(\mathcal{X},\delta)}(n-1)$. Our first observation in this article is that the branching graphs of the algebras $A_{(\mathcal{X},\delta)}(\infty), \ \mathcal{X} = \mathcal{S}, \mathcal{O}, \mathcal{B}, \mathcal{H}$ can be obtained from smaller graphs called \emph{principal graphs} through a process called \emph{pascalization} in \cite{VN06}. In all four cases, these smaller graphs are closely related to famous Young graph $\Y$, the branching graph of the infinite symmetric group $S_{\infty}$. In fact, for the infinite Brauer algebra, the principal graph is exactly the Young graph \cite{VN06}, while for the other cases some adaptations are necessary. When $\mathcal{X} = \mathcal{H}$, the principal graph looks particularly appealing as it encodes an interesting 'second order' growth model for pairs $(\lambda,\mu)$ of \emph{Young diagrams}: on may either grow the first order diagram $\mu$ by adding a box or one may grow the second order diagram $\lambda$ by adding a box that is taken from $\mu$. We call this new branching graph the \emph{coupled Young graph}. 

Our main order of business in this article is the study of the \emph{minimal boundary} of the principal graphs and their pascalizations as this boundary is known to be homeomorphic to the simplex of extremal traces on the associated direct limit algebras. The minimal boundaries of different branching graphs have been studied extensively in the literature (see e.g. \cite{BO16} \cite{VN06} \cite{VN11} \cite{Go12} \cite{Wa20} and the references therein). Since the minimal boundary can be alternatively interpreted as the simplex of extremal traces on the associated inductive limit algebra, see Theorem \ref{thm.boundarytraces}, or a simplex of probability measures on the branching graphs it is of interest from both an algebraic and a stochastic viewpoint. The most prominent and influential example of a minimal boundary computation is a result of Thoma \cite{Th64} that describes the minimal boundary of the Young graph $\Y$. Fittingly, this boundary is therefore referred to nowadays as the \emph{Thoma simplex}. It will also appear as the minimal boundary of our principal graphs.

\begin{letterthm}
The minimal boundaries of the principal graphs of the tower of algebras $A_{(\mathcal{X},\delta)}(0) \subset A_{(\mathcal{X},\delta)}(1) \subset \dots ,$  for $\mathcal{X} = \mathcal{S}, \mathcal{O}, \mathcal{B}, \mathcal{H}$ are homeomorphic to the Thoma simplex $T$. In particular, this holds true for the coupled Young graph.
\end{letterthm}

For $\mathcal{X} = \mathcal{O}$, this theorem is just a rephrasing of Thoma's theorem as the principal graph is the Young graph $\Y$ and for $\mathcal{X} = \mathcal{S}$, this result was already observed in \cite{VN06}. The remaining two cases will be proven in Sections \ref{subsec.bistochasticprincipal} and \ref{subsec.coupledYoung}. Next, we study the minimal boundaries of the pascalizations of our principal graphs, that is to say the branching graphs of our four infinite diagram algebras. In stark contrast to the results of the companion article \cite{Wa20}, where pascalization led to a much larger boundary and in particular to interesting examples of random walks on trees, we will show that pascalization will not enlarge the boundary for the case  $\mathcal{X} = \mathcal{B}$, see Section \ref{sec.boundaryrookBrauer}. A similar result was already obtained for $\mathcal{X} = \mathcal{S}, \mathcal{O}$ in \cite{VN06}.

\begin{letterthm}
The minimal boundary of the branching graph of the tower of rook-Brauer algebras $\mathrm{rBr}_{\delta}(0)= \mathrm{rBr}_{\delta}(1) \subset \dots$ is fully supported on its principal graph and thus homeomorphic to the Thoma simplex.
\end{letterthm}

We conjecture that the same result is true in the remaining case $A_{(\mathcal{H},\delta)}(\infty)$ and we prove that this conjecture would be implied by a set of inequalities on a recursively defined array of integers $K(n,k,l), \ n= 0,1,\dots , \ 2k+l \leq n$, see Section \ref{subsec.boundpasccoupledYoung}. Since these inequalities can be checked on a computer for small values of $n$ (which we have done for $n \leq 20$), this provides some empirical evidence to back up our conjecture.

As a final result, we show in Section \ref{sec.halfliborth} that the infinite diagram algebra dual to the \emph{half-liberated orthogonal group} in Banica's and Speicher's theory of partition (or easy) quantum groups is exactly the \emph{infinite walled Brauer algebra} studied in \cite{Ni07} \cite{VN06}. We therefore conclude from \cite{VN06} that the minimal boundary of its branching graph is homeomorphic to two copies $T \times T$ of the Thoma simplex, see Theorem \ref{thm.boundaryhalliborth}. The parallel problem for other examples of halfliberated diagram algebras remains open as their representation theory is largely unknown.

We end this article with a short appendix, see Appendix \ref{Append.A}, where we add a few remarks on a tower of algebras whose branching graph is the coupled Young graph. In particular we compute their dimensions of these algebras and we suggest a presentation in terms of generators and relations.     


\paragraph*{Acknowledgements}
I am grateful to A. Bufetov, A. Freslon, P. Tarrago and M. Weber for their thoughful remarks and suggestions. I thank G. Olshanski for pointing out the article \cite{VN11} to me. 
This work was supported by the Deutsche Forschungsgemeinschaft (DFG, German Research Foundation) under Germany’s Excellence Strategy – EXC 2047 “Hausdorff Center for Mathematics”.

\section{Preliminaries on branching graphs and their minimal boundaries}

\subsection{Branching graphs and pascalization}

A \emph{branching graph} or \emph{Bratelli diagram} $\Gamma$ is a bipartite locally finite graded rooted graph, meaning that its set of vertices can be subdivided into levels $\sqcup_{n = o}^N \Gamma_n, \ N \in \N \cup \{ \infty \}$ with $|\Gamma_0| = 1$ and that edges can only connect vertices of adjacent levels. The unique vertex in $\Gamma_0$ is called the root of the graph and denoted by $\emptyset$. All concrete examples of branching graphs in this article will have infinitely many level sets $\Gamma_n$ that will themselves be finite, i.e. $|\Gamma_n|< \infty$. If two vertices $\gamma \in \Gamma_n$ and $\tilde{\gamma} \in \Gamma_{n+1}$ are connected by an edge, we will often write $\gamma \nearrow \tilde{\gamma}$. We will also use the notation $\dim_{\Gamma}(\gamma,\xi)$ for the number of paths on $\Gamma$ leading from $\gamma \in \Gamma_n$ to $\xi \in \Gamma_m, \ m > n$ and we will shorten $\dim_{\Gamma}(\emptyset,\gamma)$ to $\dim_{\Gamma}(\gamma)$. \\ 
\\
Branching graphs encode the induction/restriction rules of inductive sequences $A_0 = \C \hookrightarrow A_1 \hookrightarrow A_2 \hookrightarrow \dots $ of finite-dimensional semisimple algebras. To recall how one gets a branching graph from such an inductive sequence, we introduce the following notation. If $A \subset B$ is an inclusion of semisimple algebras and $M$ is a left $B$-module, we denote by $M^{\downarrow}$ the left $A$-module obtained by restricting the action of $B$ to $A$.

\begin{definition} \label{def.indrestrgraph}
Let $A_0 = \C \hookrightarrow A_1 \hookrightarrow A_2 \hookrightarrow \dots $ be an inductive sequence of finite-dimensional semisimple algebras. The \emph{induction/restriction graph} of $(A_n)_{n \geq 0}$ is the branching graph $\Gamma$ defined in the following way.
\begin{itemize}
\item The $n$-th level vertex set $\Gamma_n$ is the set of (equivalence classes of) simple modules of $A_n$.
\item There are exactly $\mult(V, W^{\downarrow})$ edges between the simple $A_n$-module $V \in \Gamma_n$ and the simple $A_{n+1}$-module $W \in \Gamma_{n+1}$ where  $\mult(V, W^{\downarrow})$ is the multiplicity of $V$ in the decomposition of $W^{\downarrow}$ into simple $A_n$-modules.
\end{itemize}
\end{definition}

The branching graphs that we will discuss in this article will often be generated by smaller ones through a process called \emph{pascalization} in \cite{VN06}. See also \cite{Wa20} for more examples of pascalized graphs.

\begin{definition} \label{def.pascalization}
Let $\Gamma$ be a branching graph. The \emph{pascalization} $\cP(\Gamma)$ of $\Gamma$ is defined in the following way.
\begin{itemize}
\item The vertex set of level $n$ is $\cP(\Gamma)_n = \{ (n,\gamma) \ ; \ \gamma \in \Gamma_k, \ k \leq n, \ k \equiv n \mod 2 \}$. In particular, there is a natural projection $\pi: (n,\gamma) \mapsto \gamma$ from vertices of $\cP(\Gamma)$ to vertices of $\Gamma$.
\item In $\cP(\Gamma)$, the number of edges between $(n,\gamma) \in \cP(\Gamma)_n$ and $(n+1,\tilde{\gamma}) \in \cP(\Gamma)_{n+1}$ is the number of (undirected) edges between $\gamma \in \Gamma_k$ and $\tilde{\gamma} \in \Gamma_l$ in the original graph $\Gamma$. This number is only non-zero if $\gamma$ and $\tilde{\gamma}$ are  on neighbouring levels of $\Gamma$, that is $|k-l|=1$.
\end{itemize}
\end{definition}
.
The name \emph{pascalization} is motivated by the example of the Pascal graph, which can be understood as the pascalization of the graph of integers $\Z$ with $\N$-grading $n \mapsto |n|$ and edges connecting neighboring integers. The pascalization method is also a common tool in subfactor theory, see e.g. \cite{GHJ89}, although the terminology is inverse: if $\Lambda = \cP(\Gamma)$, one calls the smaller graph $\Gamma$ the \emph{principal graph} of $\Lambda$.

\subsection{The minimal boundary of a branching graph}

Branching graphs are a heavily used tool in asymptotic representation theory as the representation theory of inductive limit groups such as $S(\infty)$ or $U(\infty)$ can be understood through the study of certain measures on the space of infinite paths on their branching graphs. We refer the interested reader to the very nice textbooks \cite{BO16} \cite{Me17} for more information on these groups. \\ \\
For a branching graph $\Gamma$, let us denote by $(\Omega,\cF) = (\Omega_{\Gamma},\cF_{\Gamma})$ the space 
\[ \Omega = \{ \gamma_0 \nearrow \gamma_1 \nearrow \gamma_2 \nearrow \dots \} \subset \prod_{n\geq 0} \Gamma_n \]
of infinite paths on $\Gamma$ and by $\cF$ the restriction of the product $\sigma$-algebra  on $\prod_{n\geq 0} \Gamma_n$ to $\Omega$. 

\begin{definition} \label{def.centralmeasure}
A probability measure $\bP$ on $(\Omega,\cF)$ is called \emph{central} if for all $n\geq 0$, $\gamma \in \Gamma_n$, and every path $\gamma_0=\emptyset \nearrow \gamma_1 \nearrow \dots, \gamma_n=\gamma$ from the root to $\gamma$, we have 
\[\bP \left(\left\lbrace\omega=(\omega_0 \nearrow \omega_1 \nearrow \dots) \in \Omega \ ; \ \omega_1 = \gamma_1 ,\dots, \omega_n= \gamma \right\rbrace \right) = \frac{\bP(\{ \omega_n = \gamma \})}{\dim_{\Gamma}(\gamma)}. \]
The measure $\bP$ will be called \emph{ergodic} if the sets in $\cF$ that are invariant under changes of at most finitely many steps in paths, have $\bP$-measure $0$ or $1$.
\end{definition}  

The topological space of ergodic central measures on $(\Omega,\cF)$ (equipped with the weak topology) is called the \emph{minimal, Vershik-Kerov} or \emph{exit boundary} of $\Gamma$ and we will denote it by $\partial \Gamma$. An arbitrary central probability measure $\bP'$ can always be decomposed into its ergodic components, that is to say, there exists a (unique up to nullsets) probability measure $\mu$ on $\partial \Gamma$ such that $\bP'(A) = \int_{\partial \Gamma} \bP(A) \ d \mu(\bP)$. This follows for instance from the fact that the set $\cM_c(\Gamma)$ of central probability measures on $(\Omega,\cF)$ forms a Choquet simplex whose extremal points are given by the minimal boundary $\partial \Gamma$. The following theorem is a standard tool in asymptotic representation theory, see e.g.  \cite{BO16}.

\begin{theorem} \label{thm.boundarytraces}
Let $\C = A_0 \subset A_1 \subset \dots$ a sequence of finite-dimensional semisimple algebras with inductive limit algebra $A_{\infty}$. Further, let $\Gamma$ be the  branching graph of this sequence, so that the vertices $v \in V_n$ correspond to the irreducible summands of $A_n$ in its decomposition into matrix algebras. Then, the Choquet simplex of tracial states on $A_{\infty}$ is homeomorphic to $\cM_c(\Gamma)$. More precisely, the homeomorphism sends the measure $\bP$ to the tracial state $\tau_{\bP}$ determined by
\begin{align*}
\tau_{\bP}(x) = \sum_{i=1}^m \bP(X_n = v_i) \frac{\tau_i(x)}{\dim_{\Gamma}(v_i)} \qquad (x \in A_n),
\end{align*}
where $\tau_i$ is the extremal trace on $A_n$ given by the standard trace on the irreducible summand $v_i$. Under this homeomorphism, the pure tracial states are mapped to the elements of the boundary $\partial \Gamma$.
\end{theorem} 

Vershik and Kerov \cite{VK81} \cite{VK82} developed a general strategy for computing the minimal boundary of a branching graph $\Gamma$. Their approach is captured in the following theorem that is often refered to as the \emph{(Vershik-Kerov-)ergodic method}. 

\begin{theorem} \label{thm.ergodicmethod}
Let $\bP$ be an ergodic central measure on $(\Omega_{\Gamma},\cF_{\Gamma})$. Then, the set $S$ of paths $\gamma_0 \nearrow \gamma_1 \nearrow \gamma_2 \nearrow \dots$ for which the limit 
\begin{align*}
\lim_{n \to \infty} \frac{\dim_{\Gamma}(\lambda,\gamma_n)}{\dim_{\Gamma}(\gamma_n)}
\end{align*}
exists for every vertex $\lambda \in \Gamma_m, \ m \geq 0$ has $\bP$-measure $\bP(S)=1$. Moreover, the limit above is equal to $\bP(\{ \omega \in \Omega_{\Gamma}  \ ; \ \omega_m = \lambda \})$.
\end{theorem}

Concerning the minimal boundary, a useful interpretation of the pascalized graph $\cP(\Gamma)$ of which we made abundant use in \cite{Wa20} is the following. Any infinite path on $\cP(\Gamma)$ can be interpreted as an infinite walk on the base graph $\Gamma$ that starts at the root. In the context of \cite{Wa20}, where the underlying base graphs were trees, this allowed us to relate the computation of the minimal boundary of their pascalizations to the theory of random walks on trees. In this article the underlying principal graphs are related to the branching graph of the infinite symmetric group (the Young graph $\Y$) and we will not make use of this interpretation very much but we nevertheless consider it to be a good intuition boost when thinking about pascalized branching graphs.

\subsection{The boundary of the Young graph}

The most prominent and most studied example of a branching graph is the branching graph of the inclusion of symmetric groups $S_1 \subset S_2 \subset S_3 \subset \dots $. It is commonly refered to as the \emph{Young graph} $\Y$ as its $n$-th level vertex set $\Y_n$ consists of all \emph{integer partitions} $\lambda =  (\lambda_1 \geq \lambda_2 \geq \dots \geq \lambda_{l(\lambda)}), \ \lambda_1 + \dots + \lambda_{l(\lambda)} = n, \lambda_i \in \N$ which can be represented graphically as \emph{Young diagrams} and this is how we will be thinking about them. A Young diagram $\lambda \in \Y_n$ is connected to $\mu \in \Y_{n+1}$ if and only if $\mu$ is obtained from $\lambda$ by adding a box and we will write this operation by $\mu = \lambda + \Box$. The minimal boundary of $\Y$ is in one-to-one correspondence with the set of extremal traces on $\C[S_{\infty}]$ and it its concrete description is due to Thoma \cite{Th64}. In addition to Thoma's original one, many more proofs of Thoma's result are known nowadays, see \cite{BuGo15} for a nice overview. To state Thoma's theorem,  we define a \emph{Thoma parameter} to be a collection of nonnegative numbers $(\alpha;\beta) =((\alpha_n)_{n \geq 1};(\beta_n)_{n \geq 1})$ such that
\begin{align*}
\alpha_1 \geq \alpha_2 \geq \dots \geq 0, \quad \beta_1 \geq \beta_2 \geq \dots \geq 0, \quad \text{and } \quad \sum_{n=1}^{\infty} (\alpha_n + \beta_n) \leq 1.
\end{align*}
The set of Thoma parameters is called the \emph{Thoma simplex} and will be denoted by $T$. It is in fact a Choquet simplex with the product topology inherited from the embedding $T \subset [0,1]^{\infty} \times [0,1]^{\infty}$.

\begin{theorem}[Thoma's theorem] \label{thm.Thoma}
The simplex of extremal traces on $\C[S_{\infty}]$ is homeomorphic to the Thoma simplex $T$. The homeomorphism takes the Thoma parameter $(\alpha;\beta)$ to the trace $\tau_{(\alpha;\beta)}$ on $\C[S_{\infty}]$ that is uniquely determined by the following values on $S_{\infty}$:
\begin{align*}
\tau_{(\alpha;\beta)}(\sigma) = \prod_{c \text{ cycle of } \sigma} \sum_{i=1}^{\infty} \alpha_i^{l(c)} + (-1)^{l(c)-1} \sum_{i=1}^{\infty} \beta_i^{l(c)}, \quad \sigma \in S_{\infty}\backslash \{e \}.
\end{align*}
Here $l(c)$ denotes the length of a cyclic permutation $c$, that is the minimal number of transpositions needed to generate $c$, e.g. $l((1 \ 2 \ 3)) = 2$.
\end{theorem}

\section{Preliminaries on diagram algebras}

The branching graphs to be discussed in this article will always be derived from towers of finite-dimensional diagram algebras whose basis will be given by \emph{partition diagrams}. These algebras originate from Schur-Weyl duality of compact groups and have been used in \cite{BS09} by Banica and Speicher to introduce a class of compact quantum groups which they called \emph{easy}. Easy (a.k.a. partition) quantum groups have the subject of intense study ever since, see e.g \cite{We13} \cite{RaWe16}, and three important subclasses have been classified: the \emph{free} easy quantum groups, the \emph{half-liberated} easy quantum groups and the easy (classical) groups. Our focus will mostly lie on the diagram algebras associated to the easy classical groups although we will also discuss a result on the halfliberated orthogonal group in Section \ref{sec.halfliborth}. The diagram algebras associated to the free easy quantum groups have been treated in \cite{Wa20}. 

\subsection{Categories of partitions and their diagram algebras}

By a set partition with $k$ upper and $l$ lower points, we mean a decomposition of the set
$\{1,2,\dots,k,1',\dots,l'\}$ into disjoint subsets (the \emph{blocks} of the set partition) whose union is the full set $\{1,2,\dots,k,1',\dots,l'\}$. A diagrammatic depiction of a set partition can be found in Figure \ref{fig.partition}. The set of all set partitions with $k$ upper and $l$ lower points will be denoted by $\Part(k,l)$ and the set of all non-crossing (or planar) set partitions  with $k$ upper and $l$ lower points will be denoted by $\mathrm{NC}(k,l)$. Whenever the context allows it, we will refer to set partitions simply as partitions since this is the more common terminology. However, since \emph{integer partitions} will also play a role in this article, we will use the term set partitions whenever there is danger of confusion. 

\begin{figure}[h!]
\begin{center}
\includegraphics[scale=0.3]{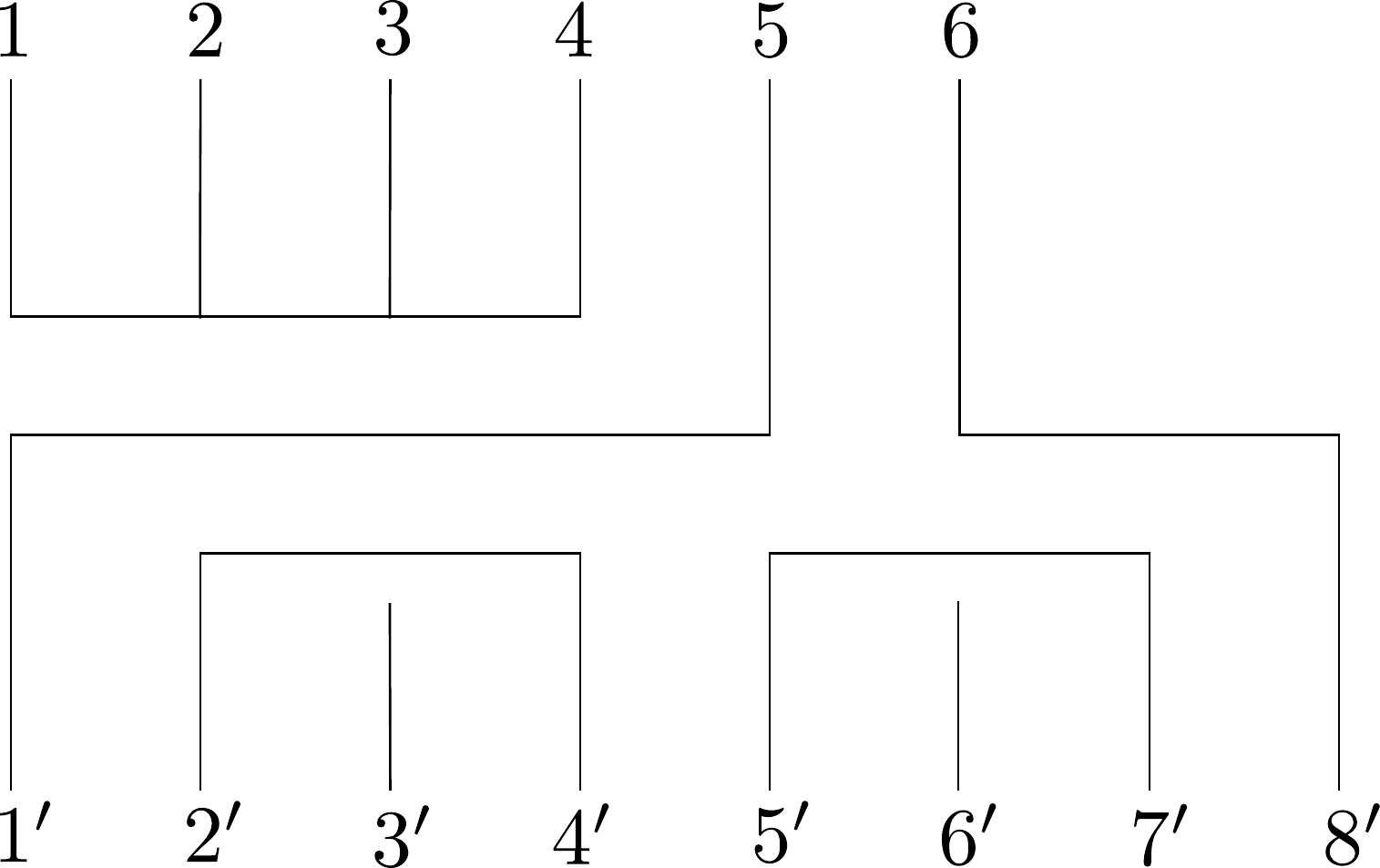}
\end{center}
\caption{\label{fig.partition} A noncrossing set partition with $6$ upper and $8$ lower points.}
\end{figure}

\begin{definition}
A \emph{category of partitions} $\cC$ is a collection $(\cC(k,l))_{k,l \in \N}$ of subsets $\cC(k,l) \subset \Part(k,l)$ such that
\begin{itemize}
\item $\cC(1,1)$ contains the identity partition that connects the upper and the lower point.
\item the family is invariant under the following category operations: tensor product (i.e. horizontal concatenation of diagrams), rotation, involution (i.e. reflecting a diagram along a horizontal line in the middle) and composition (i.e. vertical concatenation of compatible partitions $p_1 \in \cC(k,l), \ p_2 \in \cC(l,m)$). See \cite[Definition 1.4]{We13} for more details on these operations.
\end{itemize} 
\end{definition}  

The $k$-th diagram algebra $A_{(\cC,\delta)}(k)$ of the category of partitions $(\cC,\delta)$ with loop parameter $\delta$ will be the free vector space with basis $\cC(k,k)$. Multiplication of two basis vectors $e_{p_1}, e_{p_2}$ indexed by partitions $p_1$ and $p_2$ is implemented by the composition of partitions mentioned above. A diagramatic description of this operation is as follows: one draws $p_2$ on top of $p_1$ and connects blocks of $p_1$ and $p_2$ that meet in the middle. Upper points of $p_2$ and lower points of $p_1$ that are now connected to each other form a block of the partition $p_1 \cdot p_2$. We also erase closed loops that appear in the middle of the picture that do not connect to any upper or lower points. In the algebra $A_{(\cC,\delta)}(k)$ we then multiply by the loop parameter $\delta$ for each such loop. In other words, if the diagram obtained after erasing loops is $p_1 \cdot p_2$, then 
\[ e_{p_1}\cdot e_{p_2} = \delta^{\# erased \ loops} e_{p_1 \cdot p_2}. \]
As mentioned above, the involution $p^*$ of a diagram $p$ is obtained by reflecting $p$ along a horizontal line in the middle, whence we get an involution $e_p^* = e_{p^*}$ on $A_{(\cC,\delta)}(k)$. 

Recall that a finite-dimensional algebra $A$ is semisimple if and only if it possesses a \emph{positive} involution, i.e. a conjugate linear map $*:A \to A, x \mapsto x^*$ such that $(x^*)^* = x$, for which $x^*x=0$ implies $x=0$ (positivity), see e.g. \cite[Appendix II]{GHJ89}. In addition, if there is a positive involution on $A$, there is also a unique $C^*$-norm on $A$ turning it into a $C^*$-algebra. Lastly, if $A \subset B$ is an inclusion of finite-dimensional semisimple algebras, for any positive involution on $A$, there is a positive involution on $B$ extending it. Therefore, when given an inductive sequence of finite-dimensional semisimple algebras, we can always consider its inductive limit in the category of $C^*$-algebras. 

\begin{definition}
Let $(\cC,\delta)$ be a category of partitions at loop parameter $\delta$ and assume that for all $k\geq 1$, $A_{(\cC,\delta)}(k)$ is semisimple.  Let $\alpha_k: A_{(\cC,\delta)}(k) \to A_{(\cC,\delta)}(k+1)$ be the embedding obtained by mapping $e_p \to e_{p'}$, where $p'$ is obtained from $p$ by adding a through-string on the right of the diagram. The inductive limit (in the category of $C^*$-algebras) w.r.t. these embeddings will be denoted $A_{(\cC,\delta)}(\infty)$.   
\end{definition}

Note that the set $\cC(\infty):= \bigcup_{k \geq 1} \cC(k,k)$ constitutes a natural basis for $A_{(\cC,\delta)}(\infty)$.

\subsection{Diagram algebras with a crossing and easy groups} \label{sec.groups}
In \cite{BS09}, Banica and Speicher used the Tannaka-Krein duality theorem of Woronowicz \cite{Wo88} to construct for every category of partition $\cC$ at loop parameter $\delta = n$ a compact quantum group $\G(\cC,n)$ generated by a corepresentation $u$ of dimension $n$, whose endomorphism algebra $\End(u^{\ot k}) \subset B((\C^n)^{\ot k})$ is the image of a representation of $A_{(\cC,\delta)}(k)$ on $(\C^n)^{\ot k}$. See for instance the textbook \cite{Ti08} for further information on compact quantum groups.
The first class of partition quantum groups considered in \cite{BS09} is the class of categories of partitions containing the simple crossing partitions $\{\{1,2'\}, \{ 2,1' \}\} \in \cC(2,2)$. These are the categories of partitions whose associated partition quantum groups are proper groups. Our goal in this paragraph is to classify the traces on the inductive limit algebra $A_{(\cC,\delta)}(\infty)$ for the six categories with the simple crossing found in \cite[Section 2]{BS09}.  

\begin{theorem} \label{thm.classeasygroups}
There are exactly six categories of partitions containing the simple crossing, namely
\begin{itemize}
\item the category $\mathcal{S}$ of all partitions;
\item the category $\mathcal{O}$ generated by all pair partitions (i.e. partitions with blocks of size two);
\item the category $\mathcal{H}$ generated by all partitions with blocks of even size;
\item the category $\mathcal{B}$ generated by all partitions with blocks of size one or two;
\item the category $\mathcal{S}'$ generated by all partitions with an even number of blocks of odd size;
\item the category $\mathcal{B}'$ generated by all partitions with any number of blocks of size two and an even number of blocks of size one;
\end{itemize}
 At the loop parameter value $\delta = n$, the corresponding partition groups are (in the same order as above) the symmetric group $S_n$, the orthogonal group $O_n$,  the hyperoctahedral group $H_n$, the bistochastic group $B_n$, the group $S'_n = \Z_2 \times S_n$, and the group $B'_n = \Z_2 \times B_n$.
\end{theorem}

As mentioned in the introduction, the categories $\mathcal{S}$ and $\mathcal{S}'$ yield the same diagram algebras since $\mathcal{S}(k,k) = \mathcal{S}'(k,k)$ for all $k \geq 1$ (and the same is true for $\mathcal{B}$ and $\mathcal{B}'$). The algebras $A_{(\mathcal{S},n)}(k)$, $A_{(\mathcal{O},n)}(k)$, $A_{(\mathcal{B},n)}(k)$ are commonly known in the algebraic literature as the \emph{partition algebra}, the \emph{Brauer algebra} and the \emph{rook-Brauer algebra} respectively, see e.g \cite{MaS94} \cite{Ma96} \cite{Ma00} \cite{Jo94}\cite{Wen88} \cite{dMH13}. 

At loop parameter $\delta=n \in \N$, the representations $A_{(\cC,n)}(k) \to B((\C^n)^k)$ used in \cite{BS09} to describe the intertwiner spaces of the above categories are not injective when $k > n$. Also, in this case, $A_{(\cC,n)}(k)$ is not semisimple and there is no inductive limit algebra  to study. However, we can avoid this issue by adjusting the loop parameter. For most other values of the loop parameter, semisimplicity of the algebras $A_{(\cC,n)}(k), \ k\geq 1$ is guaranteed by the following theorem, see e.g. \cite[Theorem 5.13(a)]{HR05}.

\begin{theorem} \label{thm.semisimplicity}
Let $A_{\delta}, \ \delta \in \C$ be a family of algebras generated by generators and relations such that the coefficients of the relations are polynomials in $\delta$. Assume further that for some $\delta_0$, $A_{\delta_0}$ is semisimple. Then $A_{\delta}$ is semisimple for all but finitely many values of $\delta$.
\end{theorem}

We will not go into detail concerning the generators and relations of the diagram algebras considered here as they can be found in the literature referenced above. For $A_{(\mathcal{H},n)}(k)$, the $k$-th centralizer of the standard representation of the hyperoctahedral group, generators and relations have been found in \cite{Or05}.

\begin{definition}
We will call the loop parameter $\delta \in \C$ \emph{generic} for $\cC$ if for all $k\geq 1$, the algebra $A_{(\cC,\delta)}(k)$ is semisimple.
\end{definition}


\begin{remark} \label{rem.genericparameter}
The exact set of generic parameter $\mathcal{G}_{\cC}$ values has been determined for all of the above algebras except for the rook-Brauer algebra. Namely, $\cG_{\mathcal O} = \C \backslash \Z$ \cite{Wen88}, $\cG_{\mathcal S} = \C \backslash \N$  \cite{MaS94} and $\cG_{\mathcal H} = \C \backslash \N$ \cite{FM20}. 
\end{remark}

For the rest of this article, we will assume any loop parameter $\delta$ to be generic unless otherwise stated. Let $S_{\infty}= \bigcup_{n\geq 1} S_n$ denote the infinite permutation group and let $\sigma_l$ denote the transposition $(l \ l+1).$

\begin{lemma} \label{lem.symquotient}
Let $\cC$ be a category of partitions containing the simple crossing $p_c$ at generic parameter $\delta \in \C$. The subspace $I_k$ spanned by all noninvertible partitions in $\cC(k,k)$ is a (two-sided) ideal in $A_{(\cC,\delta)}(k)$ for $k=1,2,\dots, \infty$. Moreover, if $q_l$ denotes the partition that places $p_c$ on the upper and lower points $l,l+1$, then
\begin{align*}
\C[S_{k}] \to A_{(\cC,\delta)}(k)/I_k, \quad \sigma_l \mapsto q_l + I_k, \quad l <k
\end{align*}
is an isomorphism of semisimple algebras for $k=1,2,\dots, \infty$.
\end{lemma}

\begin{proof}
A partition $p \in \cC(k,k)$ is invertible if and only if every block contains exactly one upper and exactly one lower point. From this it follows that the products $pq, qp$ of a noninvertible partition $q$ with any other partition $p$ remains noninvertible and we can expand this argument linearly to arbitrary elements $q\in I_k, \ p \in A_{(\cC,\delta)}(k)$, so that $I_k$ is a two-sided ideal. The isomorphism of the quotient and $\C[S_{k}]$ is then readily checked.  
\end{proof}

When $\cC = \cC_O$, we will write $\Br_{\delta}(\infty) := A_{(\mathcal O,\delta)}(\infty)$ for the \emph{infinite Brauer algebra} at parameter $\delta$. Similarly, when $\cC = \cC_S$, we write $\mathrm{P}_{\delta}(\infty) := A_{(\mathcal S,\delta)}(\infty)$ for the \emph{infinite partition algebra} and $\mathrm{rBr}_{\delta}(\infty) := A_{(\mathcal B,\delta)}(\infty)$ fo the \emph{infinite rook-Brauer algebra} at parameter $\delta$. 
The Bratelli diagram of $\Br_{\delta}(\infty)$ was computed in \cite[Theorem 3.2]{Wen88} and that of $\Part_{\delta}(\infty)$ was computed in the works of P. Martin \cite{Ma96}, \cite{Ma00}. It was pointed out in \cite{VN06} that the branching graph of $\Br_{\delta}(\infty)$ is the pascalization of the \emph{Young graph} $\mathbb{Y}$ (i.e. the branching graph of $S_{ \infty}$). A similar observation is made for the branching graph of $\mathrm{P}_{\delta}(\infty)$: it is the pascalization of a slight alteration $\bar{\bar{\mathbb{Y}}}$ of the Young graph in which every level of $\mathbb{Y}$ is repeated two times instead of just once. 

From these observations, the following is deduced in \cite[Corollary 2.12]{VN06}. 

\begin{theorem} \label{thm.Brauerpartition}
Every trace $\tau$ on $ \Br_{\delta}(\infty)$, respectively $\mathrm{P}_{\delta}(\infty)$, is the lift of a trace $\hat{\tau}$ on the quotient $\C[S_{\infty}]$. In particular, the simplex of extremal traces on these algebras is homeomorphic to the Thoma simplex $T$. 
\end{theorem}



Our goal is to prove an analogous result for the infinite rook-Brauer algebra and to reduce the analogous statement for the algebra $A_{(\mathcal H,\delta)}(\infty)$ to a simpler one.

\section{The branching graph of the rook-Brauer algebras and its boundary} \label{sec.boundaryrookBrauer}

In this section, we will study the minimal boundary of the branching graph of the tower of rook-Brauer algebras $\mathrm{rBr}_{\delta}(0)= \mathrm{rBr}_{\delta}(1) \subset \dots$ with generic parameter $\delta$. \\
This branching graph, let us call it $\Gamma_B$, was computed in \cite{dMH13} and its construction goes as follows. Denote by $\Y_n, \ n \geq 1$ the set of Young diagrams with $n$ boxes (i.e. the number of integer partitions of $n$ arranged in weakly decreasing order). Let us also add the set $\Y_0 = \{\emptyset\}$ to this list, where $\emptyset$ is the empty partition. 

The $n$-th level vertex set $V_n, \ n \geq 0$ of $\Gamma_B$ is $\bigcup_{k = 0}^n \Y_k $ which we will identify with the set $ V_n = \{n \} \times \bigcup_{k = 0}^n \Y_k$ to better keep track of the level in which a given partition $\lambda \in Y_k$ is considered to be.  There is an edge between $(n,\lambda) \in V_n, \ (n-1,\mu) \in V_{n-1}$ if and only if one of the following three conditions hold: 
\[ \lambda = \mu, \qquad \lambda = \mu + \Box \quad \text{ or } \quad \lambda = \mu - \Box. \]

The first four levels of $\Gamma_B$ are drawn below.

\begin{figure}[h!]
\begin{center}
\includegraphics[scale=0.3]{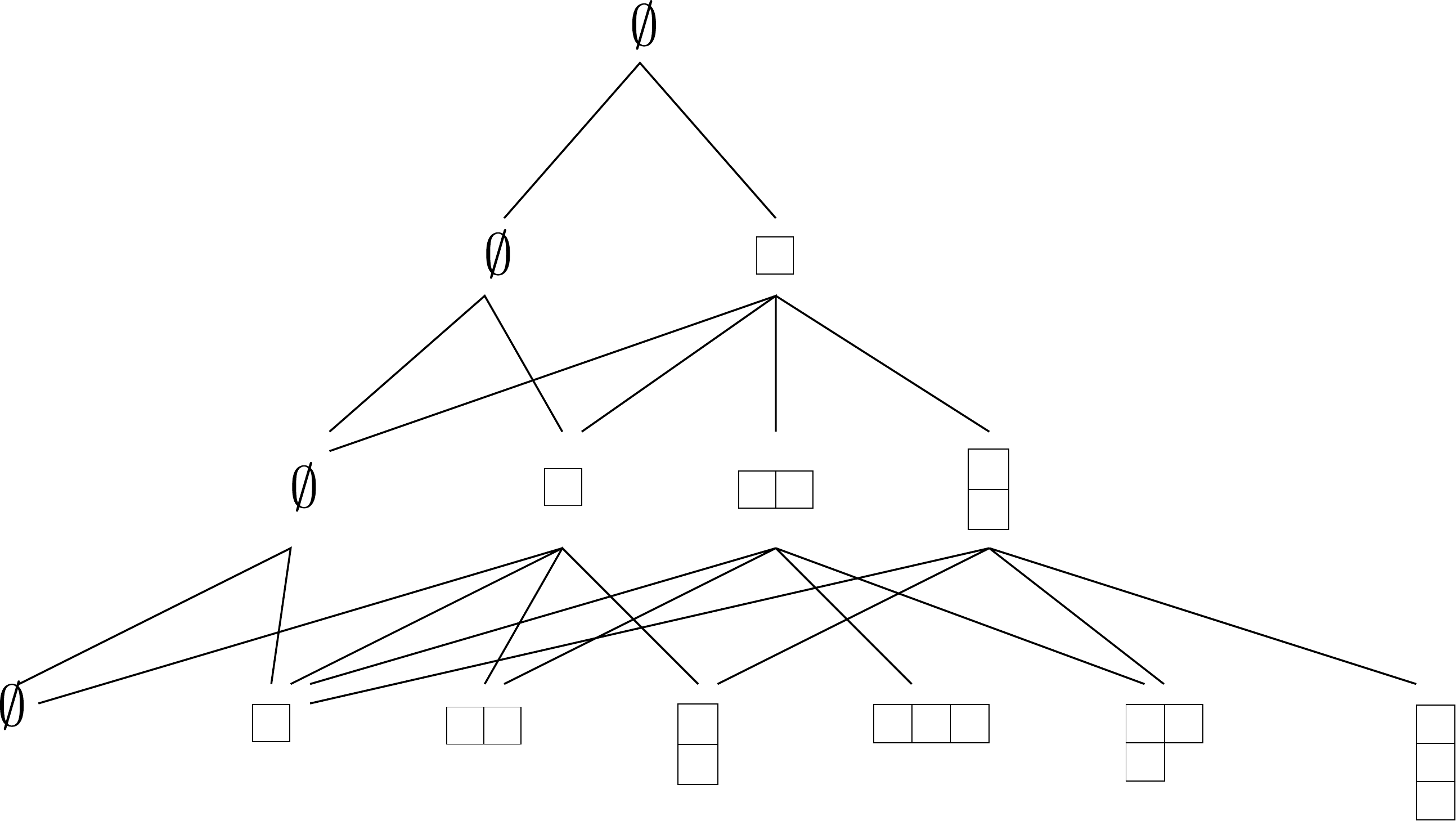}
\end{center}
\caption{\label{fig.branchingbistochastic} First four levels of $\Gamma_B$.}
\end{figure}

 We immediately observe that an infinite path on $\Gamma_B$ that starts at the root exactly corresponds to a walk on the Young graph that at every step is allowed to either jump to a neighboring vertex or to stay in place. This is analogous to the branching graph of the Brauer algebra $\Br(\infty)$ which (as the pascalization of the Young graph) describes walks on $\Y$ that jump to a neighboring vertex at every step (and are thus not allowed to remain in place).  The minimal boundary of $\Gamma_B$ therefore consists of \emph{lazy} random walks on $\Y$ (either jumping or staying) with the  following property: two walks that reach the same Young diagram $\lambda$ after $n$ steps are assigned the same probability.

By construction of the graph $\Gamma_B$, it follows that we can identify the minimal boundary $(\Omega_{\Y},\cF_{\Y})$ of the Young graph with a Borel subspace of $(\Omega_{\Gamma_B},\cF_{\Gamma_B})$. More precisely, we identify the infinite paths on $\Y$ with \emph{deterministic} walks on $\Y$, i.e. infinite paths on $\Gamma_B$ that only use the forward move $\lambda \nearrow \lambda + \Box $.

The rest of this section will be devoted to the proof the following theorem. 

\begin{theorem} \label{thm.centralmeasuresbistochastic}
Any central measure on $(\Omega_{\Gamma_B},\cF_{\Gamma_B})$ is fully supported on the Borel subspace $(\Omega_{\Y},\cF_{\Y})$.
\end{theorem}

The proof of Theorem \ref{thm.centralmeasuresbistochastic} will be divided into several lemmas.

The following lemma is an adaptation of \cite[Lemma 2.2]{VN06} that accommodates the additional possibility of resting steps. As is common, we denote by $|\lambda|$ the number of boxes of an integer partition $\lambda$, i.e. its level in the Young graph.

\begin{lemma} \label{lem.zeroprobnew}
Suppose that
\begin{align*}
\lim_{n \to \infty} \max_{|\lambda| <n} \frac{\dim_{\Gamma_{B}}(n-1,\lambda)}{\dim_{\Gamma_{B}}(n,\lambda)} = 0.
\end{align*}
Then, for every ergodic central measure $\bP$ on $\Gamma_B$ and for every vertex $(n_0,\lambda) \in V_{n_0}$ such that $|\lambda| < n_0$, we have
\[ \bP( \{ \omega \in \Omega_{\Gamma_B} \ ; \ X_{n_0}(\omega) = (n_0,\lambda)  \} ) = 0. \]
\end{lemma}

\begin{proof}
By the ergodic method (Theorem \ref{thm.ergodicmethod}) we can choose a path $(0,\emptyset) \nearrow \dots \nearrow (n,\lambda_n') \nearrow \dots$ in $\Gamma_B$ such that
\begin{align*}
\bP( \{ \omega \in \Omega_{\Gamma_B} \ ; \ X_{n_0}(\omega) = (n_0,\lambda)  \} ) = \dim_{\Gamma_{B}}(n_0,\lambda) \lim_{n \to \infty} \frac{\dim_{\Gamma_{B}}((n_0,\lambda); (n,\lambda'_n))}{\dim_{\Gamma_{B}}(n,\lambda'_n)}.
\end{align*} 
Let $n$ be large enough such that $\dim_{\Gamma_{B}}((n_0,\lambda); (n,\lambda'_n)) \neq 0$. Then there exists an $n$-step walk on $\Y$ that passes through $\lambda$ at step $n_0$ and ends at $\lambda'_n$. Since $|\lambda| < n_0$, this walk must include at least one non-forward step and thus $|\lambda'_n| < n $. Hence, $(n_0-1,\lambda)$ and $(n-1,\lambda'_n)$ are well-defined vertices of $\Gamma_B$ and we have
\begin{align*}
\dim_{\Gamma_{B}}((n_0,\lambda); (n,\lambda'_n)) &= |\{ n-n_0-\text{step walks from} \lambda \text{ to }  \lambda'_n\}| \\
&= \dim_{\Gamma_{B}}((n_0-1,\lambda); (n-1,\lambda'_n)) \\
&\leq \dim_{\Gamma_{B}}(n-1,\lambda'_n). 
\end{align*}
Therefore, our assumption implies
\begin{align*}
\lim_{n \to \infty} \frac{\dim_{\Gamma_{B}}((n_0,\lambda); (n,\lambda'_n))}{\dim_{\Gamma_{B}}(n,\lambda'_n)} = 0,
\end{align*}
whence $\bP( \{ \omega \in \Omega_{\Gamma_B} \ ; \ X_{n_0}(\omega) = (n_0,\lambda)  \} ) = 0$.
\end{proof}



\begin{remark}
\begin{itemize}
\item[(1)] Colloquially, Lemma \ref{lem.zeroprobnew} states that every central ergodic measure on $\Gamma_B$ must be fully supported on the copy of $\Y$ that it contains. If we interprete infinite paths in $\Gamma_B$ as lazy walks on $\Y$, then this means every central ergodic random lazy walk on $\Y$ must already converge deterministically towards a boundary point of $\Y$.
\item[(2)] Intuitively, the reason for this behaviour is that the Young graph $\Y$ is too connected. The centrality condition forces a measure to spread its mass over too many walks, so that in the limit no mass on nondeterministic paths is retained. This is in stark contrast to the situation where the principal graph is a tree, see \cite{Wa20}.  
\end{itemize}
\end{remark}

For notational convenience let us define $\dim_{\Gamma_B}(n,\lambda) = 0$, whenever $|\lambda| > n$.

\begin{lemma} \label{lem.recursionbistochastic}
There exists an array of integers $M(n,l), \ n,l=0,1,2,\dots$ such that 
\begin{align} \label{eqn.recursion1}
\frac{\dim_{\Gamma_B}(n,\lambda)}{\dim_{\Y} \lambda} = M(n,|\lambda|)
\end{align}
for all $n \in \N$ and all integer partitions $\lambda$ with $|\lambda| \leq n$. In particular this quotient of dimensions depends only on $n$ and  the number of boxes of $\lambda$. \\
Moreover, the array $M(n,l), \ n,l=0,1,2,\dots$ is uniquely determined by the following recursion and initial conditions:
\begin{align}
M(n,n) &= 1, \quad  n \geq 0; \label{eqn.recursion2a} \\
M(n,l) &=0,  \quad  n \geq 0, \ l >n; \\
M(1,0) &=1; \\
M(n,0) &= M(n-1,0) + M(n-1,1), \quad n \geq 1, \\
M(n,l) &= M(n-1,l-1) + M(n-1,l) + (l+1) M(n-1,l+1), \quad n \geq 1, \ l=1,\dots n-1. \label{eqn.recursion2}
\end{align}

\end{lemma}

\begin{proof}
We prove Lemma \ref{lem.recursionbistochastic} by induction with respect to the parameter $n$. The statement is clearly true for $n=0,n=1$, as the only Young diagrams reached at most one step are the empty diagram $\emptyset$ and the one box $\Box$. Let us thus assume that \ref{eqn.recursion2a}-\ref{eqn.recursion2} are true for some $n-1$. We see that
\begin{align*}
\dim_{\Gamma_B}(n,\emptyset) &= \dim_{\Gamma_B}(n-1,\emptyset) + \dim_{\Gamma_B}(n-1,\Box) \\
&= (M(n-1,0) + M(n-1,1)) \cdot \dim_{\Y} \emptyset,
\end{align*}
and similarly
\begin{align*}
\dim_{\Gamma_B}(n,\lambda) &= \sum_{\eta = \lambda - \Box} \dim_{\Gamma_B}(n-1,\eta) + \dim_{\Gamma_B}(n-1,\lambda) + \sum_{\mu = \lambda + \Box} \dim_{\Gamma_B}(n-1,\mu) \\
&= M(n-1,|\lambda|-1) \sum_{\eta = \lambda - \Box} \dim_{\Y} \eta  + M(n-1,|\lambda|) \dim_{\Y} \lambda \\
 &+ M(n-1,|\lambda|+1) \sum_{\mu = \lambda + \Box} \dim_{\Y} \mu
\end{align*}
for $\lambda \in \Y_n$. By the rules on dimension in $\Y$, we have 
\[ \sum_{\eta = \lambda - \Box} \dim_{\Y} \eta  = \dim_{\Y} \lambda \]
and 
\[ \sum_{\mu = \lambda + \Box} \dim_{\Y} \mu = \dim \Ind_{S_{|\lambda|}}^{S_{\lambda+1}} \lambda = (| \lambda | +1) \dim_{\Y} \lambda.  \]
The induction hypothesis then implies Equation \ref{eqn.recursion1} and the recursive rules \ref{eqn.recursion2a}-\ref{eqn.recursion2}.
\end{proof}

\begin{remark}
The numbers $M(n,l)$ defined by recursion \ref{eqn.recursion2a}-\ref{eqn.recursion2} as well as their analogues for the Brauer algebra in \cite{VN06} are examples of \emph{Catalan-like numbers}, see e.g. \cite{Ai99}, \cite{Ai08}.
\end{remark}

Due to the 'Pascal-like' context in which the array $M(n,l), \ n,l=0,1,\dots$ appeared, one can prove the following more explicit formulas. 

\begin{lemma} \label{lem.bistochasticeasyrec}
The array $M(n,l), \ n,l=0,1,\dots$ satisfies the following equalities
\begin{itemize}
\item[(1)] $M(0,0)= M(1,0) = 1$ and $M(n,0) = M(n-1,0) + (n-1) M(n-2,0), \ n>1 $.
\item[(2)] $M(n+l,l) = \binom{n+l}{l} M(n,0), \ l \geq 1, \ n \geq 0$.
\item[(3)] The sequence $\frac{M(n-1,0)}{M(n,0)}, n\geq 1$ is weakly decreasing.
\item[(4)] The sequence $ n \cdot \frac{M(n-1,0)}{M(n,0)}, n\geq 1$ is weakly increasing.
\end{itemize}

\end{lemma}

\begin{proof}
The proof of this result is a straithforward inductive argument. We will omit the easy arguments for parts (1) and (2) and immediately addres part (3) and (4). Set $a_n = \frac{M(n-1,0)}{M(n,0)}, \ n=1,2,\dots.$ Dividing the recursion of part (1) by $a_n$, we see that this sequence satisfies the recursion
\begin{align*}
a_1 = 1, \qquad a_n = \frac{1}{1+(n-1)a_{n-1}}, \quad n=2,3,\dots.
\end{align*}
We now prove by induction that 
\begin{align} \label{eq.inductioninequal}
a_{n+1} \leq a_n \leq \frac{n+1}{n} a_{n+1} \qquad n=1,2,\dots.
\end{align}
As $a_2 = 2$, this is indeed correct for $n=1$. Let us assume that  \ref{eq.inductioninequal} holds for all $1\leq k \leq n$ for some $n \in \N$. For $n+1$, we then see that
\begin{align*}
a_{n+2} = \frac{1}{1+(n+1)a_{n+1}} \leq \frac{1}{1+(n+1) \tfrac{n}{n+1}a_{n}} = \frac{1}{1+n a_{n}} = a_{n+1} 
\end{align*}
by the induction hypothesis. For the other inequality, we  note first that the hypothesis implies that $a_{n+1} \leq 1$. We thus get
\begin{align*}
a_{n+2} &= \frac{1}{1+(n+1)a_{n+1}} \\
&=  \frac{1+na_n}{n+2 +a_n} \\
&\geq \frac{n(1+na_{n+1})}{n^2+2n +(n+1) a_{n+1}} \\
&\geq \frac{n(1+n)}{n^2+3n+1} a_{n+1} \\
&\geq \frac{n+1}{n+2} a_{n+1}. 
\end{align*}
Hence, the lemma is proven.
\end{proof}

\begin{corollary} \label{cor.bistochasticlimit}
 We have 
\begin{align*}
\max_{l: \ l <n} \frac{M(n-1,l)}{M(n,l)}  = \frac{M(n-1,0)}{M(n,0)} \quad \to \quad 0 \qquad \text{ as } \quad n \to \infty.
\end{align*}
\end{corollary}

\begin{proof}
Using parts (2) and (3) of Lemma \ref{lem.bistochasticeasyrec}, it follows that 
\begin{align*}
\max_{l: \ l <n} \frac{M(n-1,l)}{M(n,l)} = \frac{1}{n} \max_{k=1,\dots, n} k \cdot \frac{M(k-1,0)}{M(k,0)} = \frac{M(n-1,0)}{M(n,0)}.
\end{align*}
Since the sequence $\frac{M(n-1,0)}{M(n,0)}$ is weakly decreasing by part (4) of Lemma \ref{lem.bistochasticeasyrec}, it has a limit $a \in [0,1)$. Suppose that $a>0$. Then
\begin{align*}
M(n,0) = \prod_{k=1}^n \frac{M(k-1,0)}{M(k,0)} \leq \left(\frac{1}{a} \right)^n
\end{align*}
for all $n \geq 1$. Since $M(n,0) = \dim_{\Gamma_B}(n,\lambda)$ counts the number of $n$-step walks on the Young graph with resting steps, we have  $M(n,0) \geq \tilde{M}(n,0)=\dim_{P(\Y)}(n,0)$, the number of $n$-step walks on the Young graph without resting steps. By the proof of \cite[Theorem 2.10]{VN06}, for every $K>0$, there exists a constant $C(K)>0$ such that $\tilde{M}(2n,0)> C(M) M^n$ for all $n$. Therefore, choosing $K = \frac{2}{a^2} $, we get the contradiction
\begin{align*}
\left(\frac{1}{a} \right)^{2n} \geq M(2n,0) \geq \tilde{M}(2n,0) > C(K) 2^n \left( \frac{1}{a}\right)^{2n} > \left( \frac{1}{a}\right)^{2n}
\end{align*}
for large enough $n$, whence $a=0$.
\end{proof}

\begin{proof}[Proof of Theorem \ref{thm.centralmeasuresbistochastic}]
Theorem \ref{thm.centralmeasuresbistochastic} follows now directly from Lemma \ref{lem.zeroprobnew}, Lemma \ref{lem.recursionbistochastic} and Corollary \ref{cor.bistochasticlimit}.
\end{proof}



\subsection{Interpretation of $\Gamma_B$ as pascalized graph} \label{subsec.bistochasticprincipal}

There is also a second interpretation of $\Gamma_B$ as the pascalization $\cP(\Lambda)$ of a principal graph $\Lambda$. The $n$-th level vertex set of $\Lambda$ is $\Lambda_n = V_{n} \cup V_{n-1}$ and there is an edge between $\lambda \in \Lambda_n$ and $\mu \in \Lambda_{n-1}$ if and only if $\lambda = \mu$ or $\lambda = \mu + \Box$. Note that every vertex $\lambda \in \Lambda_n$ has exactly one edge to all its neighbours in $\Y$ and one edge to a copy of itself, so that indeed $\cP(\Lambda) = \Gamma_B$.

We can compactly draw the graph $\Lambda$ in the following way. Consider the branching graph $S(\Lambda)$ (which we will ad-hoc call the \emph{skeleton} of $\Lambda$) whose $n-th$ level vertex sets are $S(\Lambda)_0 = \{ Y_0 \}, S(\Lambda)_n = \{ \Y_n, \Y_{n-1} \}$  with one edge connecting $Y_{n-1} \in S(\Lambda)_n$ to $Y_{n} \in S(\Lambda)_{n+1}$ and one edge from $Y_{n} \in S(\Lambda)_n$ to $Y_n, Y_{n+1} \in S(\Y)_{n+1}$ each. A single edge from $\Y_n$ to $\Y_{n+1}$ in $S(\Lambda)$ then represents the set of edges from $\Y_n$ to $\Y_{n+1}$ in the Young graph and the edge from $\Y_n$ to $\Y_{n}$    in $S(\Lambda)$ represents the set $\{ (\lambda,\lambda), \ \lambda \in \Y_n \}$.\\

Define the family of finite-dimensional semisimple algebras $B_0 = \C, \ B_n = \C \oplus M_n, \ n \geq 1 $ and the embeddings
\begin{align*}
\iota'_0: B_0 \to B_1, \qquad \qquad x &\mapsto (x,x) \\ 
\iota'_n: B_n \to B_{n+1},  \qquad (x,A) &\mapsto (x, \bigl(\begin{smallmatrix}  x & 0\\  0 & A
 \end{smallmatrix}\bigl)).
\end{align*}
Then, $S(\Lambda)$ is exactly the branching graph of the inductive sequence 
\[ B_0 \hookrightarrow^{\iota'_0} B_1 \hookrightarrow^{\iota'_1} B_2 \hookrightarrow^{\iota'_2} \dots \] 
as can be easily checked. From these observations we obtain directly the following description of $\Lambda$.

\begin{lemma} \label{lem.bistochbranching}
Define the finite-dimensional $C^*$-algebras $A_0 = \C,\ A_n = \C[S_n] \oplus M_n(\C[S_{n-1}]), \ n \geq 1$ and the embeddings
\begin{align*}
\iota_0: A_0 \to A_1, \qquad \qquad x &\mapsto (x,x) \\ 
\iota_n: A_n \to A_{n+1},  \qquad (x,A) &\mapsto (x, \bigl(\begin{smallmatrix}  x & 0\\  0 & A
 \end{smallmatrix}\bigl)).
\end{align*}
Then, $\Lambda$ is the branching graph of the inductive sequence 
\[ A_0 \hookrightarrow^{\iota_0} A_1 \hookrightarrow^{\iota_1} A_2 \hookrightarrow^{\iota_2} \dots. \] 
\end{lemma}

For completeness, we will also describe the trace simplex of the principal graph $\Lambda$.

\begin{proposition} \label{prop.tracesLambda}
Let $A_{\infty}$ be the direct limit algebra of the inductive sequence $(A_n)_{n \geq 0}$ of Lemma \ref{lem.bistochbranching}. The trace simplex of $A_{\infty}$ (and thus the minimal boundary of $\Lambda$) is homeomorphic to the Thoma simplex $T$, i.e. the trace simplex of $\C[S_{\infty}]$.
\end{proposition}

\begin{proof}
First, let $\tau$ be a trace on $\C[S_{\infty}]$ and define a trace $\tau_n: A_n \to \C$ by
\[ \tau_n((x,A)) = \tau|_{\C[S_n]} (x) \qquad \text{for } (x,A) \in A_n = \C[S_n] \oplus M_n(\C[S_{n-1}]). \]
Clearly, $\tau_n \circ \iota_{n-1} = \tau_{n-1} $ on $A_{n-1}$ for all $n \geq 1$, so that $\tau$ yields a well-defined trace on $A_{\infty}$. \\
To show that every trace $\tau'$ on $A_{\infty}$ is of the above form, write 
\begin{align*}
\tau'|_{A_n} = c_n \tau_n^{(1)} + (1-c_n) (\tfrac{\tr_n}{n}\ot \tau_n^{(2)})
\end{align*}
where $c_n \in [0,1]$ and where $\tau_n^{(1)}$ is a trace on $\C[S_n]$ and $\tau_n^{(2)}$ is a trace on $\C[S_{n-1}]$. Now, choosing $x=0 \in \C[S_{n-1}]$ and $A$ as the identity $A =I \in M_{n-1}(\C[S_{n-1}])$ in the equation
\begin{align*}
c_n \tau_n^{(1)}(\iota_{n-1}(x,A)) + (1-c_n) (\tfrac{\tr_n}{n}\ot \tau_n^{(2)})(\iota_{n-1}(x,A)) = c_{n-1} \tau_{n-1}^{(1)} + (1-c_{n-1}) (\tfrac{\tr_{n-1}}{n-1}\ot \tau_{n-1}^{(2)})
\end{align*}
yields the recursion $c_0=1$ and $n c_{n-1} = (n-1)c_n +1$ for $n\geq 1$. Since this recursion forces $c_n=1$ for all $n\geq 0$, $\tau'$ is indeed of the required form. 
\end{proof}

\section{The branching graph of $A_{(\cC_H,\delta)}(\infty)$ and its boundary}

In this section, we will discuss the branching graph of the tower of algebras $A_{(\cC_H,\delta)}(0) \subset A_{(\cC_H,\delta)}(1) \subset \dots $. In the easy quantum group setup of Banica and Speicher \cite{BS09}, are dual to the hyperoctahedral group $H_n$ when $\delta =n \in \N$, although, as always, we assume the underlying loop parameter to be generic, meaning in particular that $\delta \neq n$. \\
The branching graph $\Gamma_H$ of the tower $A_{(\cC_H,\delta)}(0) \subset A_{(\cC_H,\delta)}(1) \subset \dots $ has been computed in \cite{Or05}. Its vertices at level $n$ are pairs of Young diagrams $(\lambda,\mu)$ that satisfy the condition $2|\lambda|+ |\mu| \leq n$. Two vertices $(\lambda,\mu) \in \Gamma_{H,n}, \ (\xi,\eta) \in \Gamma_{H,n+1}$ are connected by an edge if and only if they are related by one of the following moves
\begin{align*}
&(\xi,\eta) \ = \ (\lambda,\mu + \Box), \qquad \qquad (\xi,\eta) \ = \ (\lambda,\mu + \Box), \\ &(\xi,\eta) \ = \ (\lambda + \Box,\mu - \Box), \qquad (\xi,\eta) \ = \ (\lambda - \Box,\mu + \Box).
\end{align*}

Once again, the branching graph of $\Gamma_H$ can be understood as the pascalization $\Gamma_H = \cP(\Theta)$ of a smaller underlying graph $\Theta$, which we will call the \emph{coupled Young graph} and which will describe next.

\subsection{The coupled Young graph} \label{subsec.coupledYoung}

Although still related to the Young graph, the coupled Young graph is more involved than the principal graphs for the partition, Brauer and rook-Brauer algebras. It might be of independent interest as it describes an interesting growth dynamic for Young diagrams. One diagram can be thought of as a 'first order object' that can grow boxes out of an infinite reservoir of boxes whereas the 'second order' diagram can only add boxes that it 'steals' from the first order diagram.

\begin{definition} \label{def.CatalanYoung}
The \emph{coupled Young} graph $\Theta$ is the branching graph whose $n$-th level set $\Theta_n$ is the disjoint union 
\[\Theta_n = \bigsqcup_{k=0}^{\lfloor n/2 \rfloor} \Y_k \times \Y_{n-2k} \]
with an edge between $(\lambda, \mu) \in \Y_k \times \Y_{n-2k}$ and $(\lambda', \mu') \in \Y_j \times \Y_{(n+1)-2j}$ if and only if either
\begin{align*}
j=k \text{ and } \lambda'=\lambda, \ \mu' = \mu + \Box \qquad \text{or } \qquad j=k+1 \text{ and } \lambda'=\lambda +\Box, \ \mu' = \mu - \Box.
\end{align*}
\end{definition} 

The coupled Young graph has the following convenient conceptual interpretation. Draw the semi-Pascal graph with vertices on leven $n$ labelled by $\Y_k \times \Y_{n-2k}, k=1,\dots \lfloor n/2 \rfloor$ from right to left. The condition on edges in the Catalan-Young graph then means that the sets $\Y_k \times \Y_{n-2k}$ and $\Y_j \times \Y_{(n+1)-2j}$ must be connected in the Young graph combined with the 'internal' condition of only connecting (pairs of) partitions in these sets if they are connected or equal in $\Y$.

\begin{figure}[h!]
\begin{center}
\includegraphics[scale=0.4]{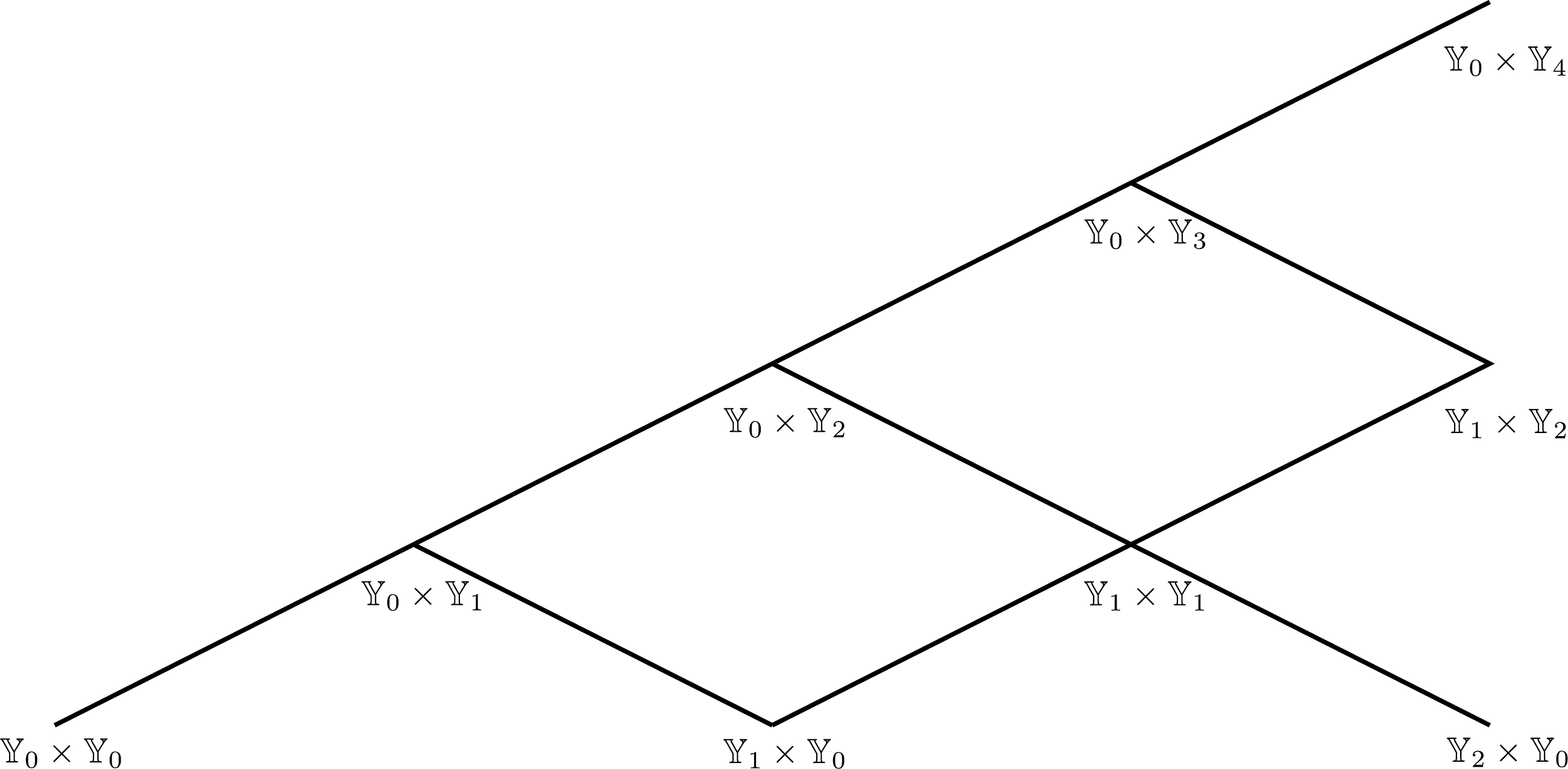}
\end{center}
\caption{\label{fig.coupledYoung} Schematic drawing of the coupled Young graph.}
\end{figure}

As alluded to above, the following is a direct consequence of the definition of the two graphs $\Gamma_H$ and $\Theta$.

\begin{proposition} \label{prop.Paschyperoctahedral}
The branching graph $\Gamma_H$ can be identified with the pascalization $\cP(\Theta)$ of the coupled Young graph $\Theta$ through $(\lambda,\mu) \in \Gamma_{H,n} \mapsto (n,(\lambda,\mu)) \in \cP(\Theta)_n$.
\end{proposition}

\begin{proposition} \label{prop.dimCYgraph}
For $(\lambda,\mu)\in \Theta_n$, we have
\begin{align*}
\dim_{\Theta}(\lambda,\mu) = \frac{1}{2^{|\lambda|}} \frac{\left(|\mu|+2|\lambda|\right)!}{|\lambda|! |\mu|!} \dim_{\Y}( \lambda) \dim_{\Y}(\mu) = \binom{n}{|\mu|} (n-|\mu|-1)!! \dim_{\Y}( \lambda) \dim_{\Y}(\mu),
\end{align*}
where $(n-|\mu|-1)!!$ is the double factorial.
\end{proposition}

\begin{remark}
The numbers $M(k,l) = \tfrac{1}{2^k} \tfrac{(l+2k)!}{k!l!} = \binom{2k+l}{l} (2k-1)!!$ appearing in Proposition \ref{prop.dimCYgraph} are related to the \emph{Bessel numbers of the first kind} $b_{k,l}$ through the formula $M(k,l) = b_{k+l,k}$ when $k\geq 1, l \geq 0$. The Bessel number of the first kind $b_{k,l}$ is the coefficient of $x^l$ in the $k$-th Bessel polynomial $Y_k$ (see the entry A001498 in the OEIS). 
\end{remark}

\begin{proof}[Proof of Proposition \ref{prop.dimCYgraph}]
We start the proof with the dimension formula
\begin{align} \label{eq.decompdimensions}
\dim_{\Theta}(\lambda,\mu) = \sum_{\eta = \lambda - \Box \atop \nu = \mu + \Box} \dim_{\Theta}(\eta,\nu) + \sum_{\nu = \mu - \Box} \dim_{\Theta}(\lambda,\nu) \qquad (\lambda,\mu) \in \Theta_n
\end{align}
that follows directly from the definition of $\Theta$. It then follows inductively that there exist numbers  $M(k,l), \ k,l \geq 0$ with $M(0,l)=1$ for all $l\geq 0$ such that
\begin{align} \label{eq.help1}
\dim_{\Theta}(\lambda,\mu) = M(|\lambda|, |\mu|) \dim_{\Y}( \lambda) \dim_{\Y}(\mu)
\end{align}
for all $(\lambda,\nu) \in \Theta_n, \ n \geq 0$. Plugging \ref{eq.help1} into Equation \ref{eq.decompdimensions}, we see that
\begin{align*}
M(|\lambda|, |\mu|) =  M(|\lambda|-1,|\mu|+1) (|\mu|+1) + M(|\lambda|,|\mu|-1),
\end{align*}
so that the numbers $M(k,l)$ are uniquely determined by
\begin{itemize}
\item  $M(0,l)=1$ for all $l\geq 0$,
\item the recursion $M(k,l) = (l+1)M(k-1,l+1) + M(k,l-1)$ for $k\geq 1, l \geq 0$.
\end{itemize}
It is now easy to verify that the numbers $M(k,l) = \tfrac{1}{2^k} \tfrac{(l+2k)!}{k!l!}$ indeed satisfy this recursion.
\end{proof}

For vertices $v,w$ of a branching graph $\Gamma$, let us denote by $\mathrm{Pa}_{\Gamma}(v,w)$ the set of paths from $v$ to $w$. \\
Let $(\lambda,\mu) \in \Theta_m$ and $(\tilde{\lambda},\tilde{\mu}) \in \Theta_n, \ n > m$. Consider paths $p_1 = (\lambda = \lambda_0 \nearrow \lambda_1\nearrow \dots \nearrow \tilde{\lambda})$ in $\mathrm{Pa}_{\Y}(\lambda,\tilde{\lambda})$ and $p_2 = (\mu = \mu_0 \nearrow \mu_1 \nearrow \dots \nearrow \mu_{n-m}) = \tilde{\mu} $ in $\mathrm{Pa}_{\cP(\Y)}(\mu,\tilde{\mu})$. Note that here, we consider $\mu$ and $\tilde{\mu}$ as vertices on the pascalized Young graph $\cP(\Y)$. Given these two paths $p_1, p_2$, we can construct a new path $p_1 \oplus p_2 := (\lambda,\mu) \nearrow (\eta_1,\nu_1) \nearrow \dots \nearrow (\eta_{n-m},\nu_{n-m}) = (\tilde{\lambda},\tilde{\mu})$ in $\mathrm{Pa}_{\Theta}((\lambda,\mu), (\tilde{\lambda},\tilde{\mu}))$ of length $n-m$ by setting
\begin{align*}
\nu_k = \mu_k \qquad \text{and } \qquad \eta_k = \lambda_{\mathrm{down}(\mu_k)}
\end{align*}
where $\mathrm{down}(\mu_k)$ denotes the number of down steps on the path $\mu_0 \nearrow \dots \nearrow \mu_k$, i.e. the number of times $1 \leq l \leq k$ for which $\mu_l = \mu_{l-1} - \Box$. Moreover, since we can decompose any path $p= (\lambda,\mu) \nearrow (\lambda_1,\mu_1) \nearrow \dots \nearrow (\lambda_{n-m},\mu_{n-m}) = (\tilde{\lambda},\tilde{\mu})$ into paths $p_1 \in \mathrm{Pa}_{\Y}(\lambda,\tilde{\lambda})$ and $p_2 \in \mathrm{Pa}_{\cP(\Y)}(\mu,\tilde{\mu})$ such that $p = p_1 \oplus p_2$, the following result is immediate.

\begin{proposition} \label{prop.pathdecompcoupledYoung}
The map 
\begin{align*}
\mathrm{Pa}_{\Y}(\lambda,\tilde{\lambda}) \times \mathrm{Pa}_{\cP(\Y)}(\mu,\tilde{\mu}) \to \mathrm{Pa}_{\Theta}((\lambda,\mu), (\tilde{\lambda},\tilde{\mu})), \qquad (p_1,p_2) \mapsto p_1 \oplus p_2
\end{align*}
is a bijection. In particular we have 
\begin{align*}
\dim_{\Theta}((\lambda,\mu);(\tilde{\lambda},\tilde{\mu})) = \dim_{\Y}(\lambda;\tilde{\lambda}) \cdot \dim_{\cP(\Y)}(\mu; \tilde{\mu}).
\end{align*}
\end{proposition}

We also remark that the path decomposition $p = p_1 \oplus p_2$ is compatible with the path structure on $\Theta$ in the following sense. Let $p = (\lambda_0,\mu_0) \nearrow (\lambda_1,\mu_1) \nearrow \dots \nearrow (\lambda_{k},\mu_{k}) $ and assume that the last move in the path is $(\lambda_k,\mu_k) = (\lambda_{k-1},\mu_{k-1} + \Box )$. If we decompose the subpath $q$ up to step $k-1$, i.e. $p = q \nearrow (\lambda_k,\mu_k)$ as $q=q_1 \oplus q_2$, then we obtain
\begin{align*}
p = q_1 \oplus \left( q_2 \nearrow \mu_k \right).
\end{align*}
Similarly, if the last move is $(\lambda_k,\mu_k) = (\lambda_{k-1} + \Box,\mu_{k-1} - \Box )$, then
\begin{align*}
p = (q_1 \nearrow \lambda_k) \oplus (q_2 \nearrow \mu_k).
\end{align*}
The observations above make it easy to determine the boundary of the coupled Young graph. First, we identify the Young graph $\Y$ as a subgraph of $\Theta$ by mapping vertices $\mu \in \Y_n$ to $(\emptyset,\mu) \in \Theta_n$ and edges $\mu \nearrow \tilde{\mu}$ to $(\emptyset,\mu) \nearrow (\emptyset,\tilde{\mu})$. As a consequence, the Thoma simplex (i.e. the minimal boundary of the Young graph) must be contained in the boundary of $\Theta$. The following theorem shows that there are no other ergodic central measures on $\Theta$.  

\begin{theorem} \label{thm.boundarycoupledYoung}
Let $\bP$ be an ergodic central measure on the space $(\Omega_{\Theta},\cF_{\Theta})$ of infinite paths on $\Theta$. Then, for every $(\lambda,\mu) \in \Theta_n$ with $\lambda \neq \emptyset$, we have 
\begin{align*}
\bP \left( \omega \in \Omega_{\Theta} \ ; \ \omega_n = (\lambda,\mu)  \right) = 0.
\end{align*}
Hence, every ergodic central measure is supported on the subgraph $\Y$, so that the minimal boundary of $\Theta$ coincides with the Thoma simplex.
\end{theorem} 

\begin{proof}
Let $(\lambda,\mu) \in \Theta_n$ with $\lambda \neq \emptyset$ as in the statement of the Theorem. By the ergodic method of Vershik and Kerov, the space of paths $(\emptyset,\emptyset) \nearrow (\lambda_1,\mu_1) \nearrow (\lambda_2,\mu_2)\nearrow \dots$ for which the limit
\begin{align*}
\lim_{n \to \infty} \frac{\dim_{\Theta}((\lambda,\mu); (\lambda_n,\mu_n))}{\dim_{\Theta}( (\lambda_n,\mu_n))}
\end{align*}
exists, has full $\bP$-measure.  We now decompose every finite subpath $(\emptyset,\emptyset) \nearrow (\lambda_1,\mu_1) \nearrow (\lambda_2,\mu_2)\nearrow (\lambda_n,\mu_n)$ of any such infinite path as explained before Proposition \ref{prop.pathdecompcoupledYoung}, so that we can rewrite the dimension formula as
\begin{align*}
\frac{\dim_{\Theta}((\lambda,\mu), (\lambda_n,\mu_n))}{\dim_{\Theta}( (\lambda_n,\mu_n))} = \frac{\dim_{\Y}(\lambda;\lambda_n)}{\dim_{\Y}(\lambda_n)} \frac{\dim_{\cP(\Y)}(\mu;\mu_n)}{\dim_{\cP(\Y)}(\mu_n)} \leq \frac{\dim_{\cP(\Y)}(\mu;\mu_n)}{\dim_{\cP(\Y)}(\mu_n)}.
\end{align*}
By the results of Vershik and Nikitin on the boundary of the pascalized Young graph, see \cite[Section 2]{VN06}, we have
\begin{align*}
\lim_{n \to \infty}\frac{\dim_{\cP(\Y)}(\mu;\mu_n)}{\dim_{\cP(\Y)}(\mu_n)} = 0,
\end{align*}
whenever $|\mu|<n$, which in our situation is equivalent to $\lambda \neq \emptyset$. Therefore any ergodic central measure must be fully supported on the copy of $\Y$ in $\Theta$ and the theorem is proven.
\end{proof}


\subsection{The boundary of $\cP(\Theta)$} \label{subsec.boundpasccoupledYoung}

In this section, we discuss central measures on the pascalized branching graph $\cP(\Theta)$. We conjecture that similar to the cases seen so far, these are fully supported on its subgraph $\Theta \subset \cP(\Theta)$. Since we have already seen in Theorem \ref{thm.boundarycoupledYoung} that every central measure on $\Theta$ is fully supported on  the Young graph $\Y$, this conjecture would even imply that the boundary of $\cP(\Theta)$ in fact reduces to the boundary of the Young graph $\Y$. We will proceed to show that our conjecture would be implied by a another conceptually very simple numerical conjecture. We have verified the latter for a significant number of cases using the algebra software SageMath but unfortunately it has resisted our proof attempts so far.

\begin{conjecture} \label{thm.boundaryreductionH}
Any central measure on  the branching graph $\cP(\Theta)$ (associated to the inductive limit algebra $A_{(\cC_H,\delta)}(\infty)$ at the generic parameter) is fully supported on $\Theta$ (and thus by Theorem \ref{thm.boundarycoupledYoung} on the Young graph $\Y \subset \Theta$).
\end{conjecture}

Conjecture \ref{thm.boundaryreductionH} would follow from
\begin{align} \label{eq.maxPascalhyp}
\lim_{n \to \infty} \max_{2|\lambda|+|\mu|<n \atop 2|\lambda|+|\mu|= n \mod 2 } \frac{\dim_{\cP(\theta)}(n-2,(\lambda,\mu))}{\dim_{\cP(\theta)}(n,(\lambda,\mu))} = 0,
\end{align} 
by \cite[Lemma 2.2]{VN06}. We will follow a strategy similar to the one for the rook-Brauer algebras in Lemma \ref{lem.recursionbistochastic} to simplify this statement. However, due to the two partitions involved in the construction of $\Theta$, we will find a recursive relation that requires an additional parameter and will thus be more difficult to analyse.  

\begin{lemma} \label{lem.recursionpaschyperoctahedral}
There exist numbers $K(n,k,l), \ n,k,l \geq 0, \ 2k+l \leq n, \ 2k+l = n \mod 2$ such that 
\begin{align*}
\dim_{\cP(\theta)}(n,(\lambda,\mu))= K(n,|\lambda|,|\mu|) \dim_{\Y}(\lambda) \dim_{\Y}(\mu)
\end{align*}
for all $(\lambda,\mu)$ with $2|\lambda|+|\mu| \leq n$ and $2|\lambda|+|\mu| = n \mod 2$. These numbers are uniquely determined by 
\begin{itemize}
\item $K(0,0,0)=1$;
\item $K(n,k,l) = \frac{1}{2^k} \frac{(2k+l)!}{k!l!}$ for all $k,l$ with $2k+l=n$;
\end{itemize}
and the recursion
\begin{align*}
K(n,k,l) \ = \ &K(n-1,k,l-1)+ (l+1)\left(K(n-1,k,l+1)+K(n-1,k-1,l+1) \right) \\
&+ (k+1)K(n-1,k+1,l-1)
\end{align*}
 for $n \geq 1$ and all $k,l$ such that $2k+l < n$ and $2k+l = n \mod 2$. 
\end{lemma}

\begin{proof}
Since $\cP(\Theta)$ is the pascalization of $\Theta$ the first two bullet points of the lemma follow from the dimension formulas for $\Theta$ of Proposition \ref{prop.dimCYgraph}.
The proof of the recursion is completely analogous to the proofs of Lemma \ref{lem.bistochasticeasyrec} and Proposition \ref{prop.dimCYgraph} with the dimension decomposition formula now being 
\begin{align*}
\dim_{\cP(\theta)}(n,(\lambda,\mu)) &= \sum_{(\xi,\eta)\nearrow_{\Theta} (\lambda,\mu)} \dim_{\cP(\theta)}(n-1,(\xi,\eta)) \quad + \sum_{(\lambda,\mu)\nearrow_{\Theta} (\xi,\eta)} \dim_{\cP(\theta)}(n-1,(\xi,\eta)) \\
&= \sum_{(\xi,\eta)=(\lambda,\mu - \Box)} \dim_{\cP(\theta)}(n-1,(\xi,\eta)) \quad + \sum_{(\xi,\eta)=(\lambda,\mu + \Box)} \dim_{\cP(\theta)}(n-1,(\xi,\eta)) \\
&+ \sum_{(\xi,\eta)=(\lambda + \Box,\mu - \Box)} \dim_{\cP(\theta)}(n-1,(\xi,\eta)) \quad + \sum_{(\xi,\eta)=(\lambda - \Box,\mu + \Box)} \dim_{\cP(\theta)}(n-1,(\xi,\eta)) .
\end{align*}
Using the identities
\begin{align*}
\sum_{\xi = \lambda - \Box} \dim_{\Y}(\xi) = \dim_{\Y}(\lambda), \qquad \sum_{\xi = \lambda + \Box} \dim_{\Y}(\xi) = (|\lambda| +1) \dim_{\Y}(\lambda),
\end{align*}
we get by induction and a straightforward computation that
\begin{align*}
\dim_{\cP(\theta)}(n,(\lambda,\mu)) \qquad \qquad\qquad\qquad\qquad \qquad \qquad \qquad\qquad\qquad\qquad \qquad\qquad \qquad\qquad\qquad\qquad \\
= ( K(n-1,|\lambda|,|\mu|-1)+ (|\mu|+1)(K(n-1,|\lambda|,|\mu|+1)+K(n-1,|\lambda|-1,|\mu|+1) ) \\
+ (|\lambda|+1)K(n-1,|\lambda|+1,|\mu|-1) ) \dim_{\Y}(\lambda) \dim_{\Y}(\mu), 
\end{align*}
which implies the recursion formula stated in the lemma.
\end{proof}

We will now show that Conjecture \ref{thm.boundaryreductionH} would be implied by the following numerical conjecture. 

\begin{conjecture} \label{lem.maximumofKs}
We have 
\begin{align*}
\frac{K(n-2,k,l)}{K(n,k,l)} \geq \max \left( \frac{K(n-2,k+1,l)}{K(n,k+1,l)}, \frac{K(n-2,k,l+2)}{K(n,k,l+2)} \right) 
\end{align*}
for all $n\geq 3$ and $k,l$ such that $2k+l < n-2$ and $2k+l = n \mod 2$. In particular
\begin{align*}
\max_{2k+l<n \atop 2k+l= n \mod 2 } \frac{K(n-2,k,l)}{K(n,k,l)} = \frac{K(n-2,0,\delta(n))}{K(n,0,\delta(n))},
\end{align*}
where $\delta(n)=0$ if $n$ is even and $\delta(n)=1$ if $n$ is odd.
\end{conjecture}

Since it is easy to implement the recursion for the numbers $K(n,k,l)$, one can check the validity of Conjecture \ref{lem.maximumofKs} for small values of $n$ and the allowed values of $k,l$ in a computer algebra program of the reader's choice. We have done so for all $n \leq 20$ in SageMath. The reason that Conjecture \ref{lem.maximumofKs} has resisted our proof attempts so far is the following. One would for instance like to compare the expressions 
\[ K(n-2,k,l) K(n,k+1,l) \geq K(n-2,k+1,l) K(n,k,l) \]
by replacing the numbers in the inequality using the recursion formula of Lemma \ref{lem.recursionpaschyperoctahedral} and an inductive argument. However a term by term comparison of the $16$ terms on both sides of the inequality turns out to be insufficient. A similar problem appears if one would like to prove Corollary \ref{cor.bistochasticlimit} for the numbers $M(n,l)$ this way, but it that case we had found an easy formula for $M(n,l)$ in Lemma \ref{lem.recursionbistochastic} to facilitate the argument. We therefore believe that a more explicit formula for the numbers $K(n,k,l)$ is necessary to prove Conjecture \ref{lem.maximumofKs}.

The following lemma shows that Conjecture \ref{lem.maximumofKs} would in fact imply Conjecture \ref{thm.boundaryreductionH}.

\begin{lemma} \label{lem.zerolimithyper}
We have
\begin{align*}
\lim_{n \to \infty} \frac{K(n-2,0,\delta(n))}{K(n,0,\delta(n))} = 0.
\end{align*}
\end{lemma}

\begin{proof}
We first note that $K(2n-1,0,1) = K(2n,0,0)$ for $n \geq 0$, so that it suffices to show that 
\begin{align*}
\lim_{n \to \infty} \frac{K(2n-2,0,0)}{K(2n,0,0)} = 0.
\end{align*}
Moreover, Frobenius reciprocity implies that $K(2n,0,0) = a_n := \dim \End_{H_m}(V^{\ot n})$ for arbitrary $m > 2n$, where $H_m = \Z_2 \wr S_m$ is the hyperoctahedral group. By \cite[Corollary 3.1]{Or05}, these dimensions satisfy the recursion
\begin{align*}
a_n = \sum_{j=1}^n \binom{2n-1}{2j-1} a_{n-j}
\end{align*}
so that 
\begin{align*}
\frac{K(2n-2,0,0)}{K(2n,0,0)} = \frac{a_{n-1}}{a_n} = \frac{1}{2n-1} \left( 1- \sum_{j=2}^n \binom{2n-1}{2j-1} \frac{a_{n-j}}{a_n} \right) \leq \frac{1}{2n-1}. 
\end{align*}
Thus, the claim is proven.
\end{proof}


The following is now immediate from the observations made in this section.

\begin{corollary}
Conjecture \ref{lem.maximumofKs} implies Conjecture \ref{thm.boundaryreductionH}.
\end{corollary}

\section{The diagram algebras of the halfliberated orthogonal quantum group} \label{sec.halfliborth}

We will now proceed to discuss the branching graph of the diagram algebras $A_{(\cO^*,\delta)}(k)$ that serve as Schur-Weyl duals to the \emph{halfliberated orthogonal group} introduced in \cite{BS09} and analysed in \cite{BV09}. These diagram algebras should be considered as an intermediate step between the Temperley-Lieb algebras $\mathrm{TL}_{\delta}(k) = A_{(O^+,\delta)}(k)$ that are dual to the free orthogonal quantum groups $O_n^+$ at $\delta=n$ and the Brauer algebras $B_{\delta}(k) = A_{(O,\delta)}(k)$ that are dual to the orthogonal groups $O_n$ at $\delta=n$. While the former is build from noncrossing pair partitions and the latter from all pair partitions, the algebras $A_{(\cO^*,\delta)}(k)$ are generated by all noncrossing pair partitions and the triple crossing \includegraphics[scale=0.09]{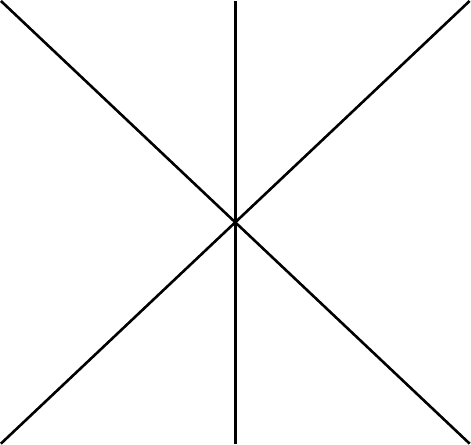}. Thus, some but not all crossings are allowed.

\begin{definition} \label{def.halfliberatedorthogonal}
The algebra $A_{(\cO^*,\delta)}(k)$ at loop parameter $\delta \in \C$ is the diagram algebra spanned by all noncrossing pair partitions on $k$ upper points $1,\dots,k$ and $k$ lower points $1',\dots, k'$ and the partitions $s_2,\dots, s_{k-1}$, where 
\[ s_j = \{ \{j-1,(j+1)' \}, \{ j,j' \}, \{ j+1,(j-1)' \} \} \cup \{ \{i,i' \}, \ \ i\neq j-1,j,j+1 \}. \]
\end{definition}

\begin{figure}[h!]
\begin{center}
\includegraphics[scale=0.25]{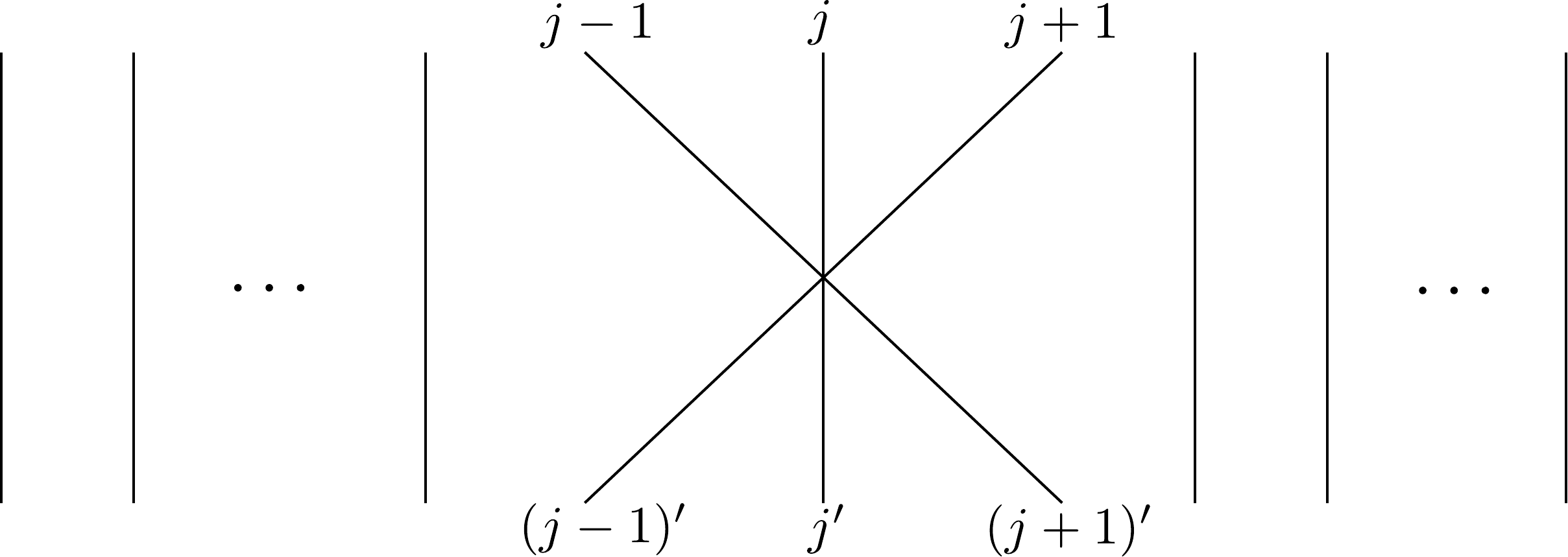}
\end{center}
\caption{\label{fig.triplecrossing}  The set partition $s_j$.}
\end{figure}

There are several other examples of halfliberated diagram algebras, see \cite{BS09} \cite{We13} but to our knowledge, the algebras $A_{(\cO^*,\delta)}(k)$ are the only ones whose representation theory is has been computed, see \cite{BV09}. For our purposes, the crucial result of \cite{BV09} is the following.

\begin{theorem}[Theorem 2.5 in \cite{BV09}]
If $V$ denotes the standard representation of the unitary group $U_n$ and $\bar{V}$ is its adjoint representation, then 
\begin{align*}
A_{(\cO^*,n)}(k) \ \cong \ \End_{U_n}(\underbrace{V \otimes \bar{V}\otimes V \otimes \dots}_{k \text{ copies}}).
\end{align*}
\end{theorem}

In particular, since $\End_{U_n}(\underbrace{V \otimes \bar{V}\otimes V \otimes \dots}_{k \text{ copies}})$ is known to be semisimple for $2k < n$, see \cite{Ni07}, the above theorem implies in combination with Theorem \ref{thm.semisimplicity} that $A_{(\cO^*,\delta)}(k)$ is generically semisimple. In \cite{Ko89} \cite{Tu89} \cite{Ni07}, these algebras appeared under the name \emph{walled Brauer algebras} and the branching graph of the inclusion $A_{(\cO^*,\delta)}(1) \subset A_{(\cO^*,\delta)}(2) \subset \dots$ at the generic parameter was discussed in \cite{Ni07} and \cite{VN06}. It the latter work, it was shown to be the pascalization of the branching graph $\bar{\Y}$ of the sequence
\begin{align*}
\C[S_0 \times S_0] \subset \C[S_1 \times S_0] \subset \C[S_1 \times S_1] \subset \C[S_2 \times S_1] \subset \C[S_2 \times S_2] \subset \dots.   
\end{align*} 
Therefore, \cite[Theorem 2.11]{VN06} yields the classification of ergodic central measures on the branching graph $\cP(\bar{\Y})$ and hence of extremal traces on the infinite diagram algebra $A_{(\cO^*,\delta)}(\infty)$.

\begin{theorem} \label{thm.boundaryhalliborth}
Every central measure on (the boundary of) $\cP(\bar{\Y})$ is fully supported on (the boundary of) $\bar{\Y}$ and therefore the simplex of ergodic central measures is in natural bijection with two copies $T \times T$ of the Thoma simplex. 
\end{theorem}

As Vershik and Nikitin also point out in \cite{VN06}, the above result precisely means that every trace on $A_{(\cO^*,\delta)}(\infty)$ is an extension of a trace on the quotient $A_{(\cO^*,\delta)}(\infty)/I_{\infty} \cong \C[S_{\infty} \times S_{\infty}]$, where $I_{\infty}$ is once again the ideal spanned by all noninvertible partitions in $A_{(\cO^*,\delta)}(\infty)$.

\appendix

\section{Algebras with coupled Young branching} \label{Append.A}

In this appendix, we collect a few observations on sequences of finite-dimensional $C^*$-algebras $A_0\subset A_1 \subset \dots$, whose branching rules are encoded by the coupled Young graph. By the general theory of Bratelli diagrams, we know that such a sequence must exist and even without detailed knowledge on their structure, we can immediately derive a formula for their dimensions from Proposition \ref{prop.dimCYgraph}. Below, we will also suggest (without proof) a presentation of algebras $A_0 \subset A_1 \subset \dots$ in terms of generators and relations whose branching graph we believe to be the coupled Young graph when the parameter $\delta$ is generic. These algebras are quotients of the centralizers of the hyperoctahedral group and our presentation is based on the presentation of the centralizer of the hyperoctahedral group suggested in \cite[Section 3.3]{Or05}.

Let us start with the formulas for the dimensions of a sequence of finite-dimensional $C^*$-algebras $A_0\subset A_1 \subset \dots$ whose branching graph is $\Theta$.   

\begin{proposition} \label{cor.algebradimensionCY}
We have 
\begin{align*}
\dim A_n = \sum_{l=0 \atop l = n\mod 2}^n \frac{\left( n! \right)^2}{2^{n-l} \cdot l! \cdot  \left(\tfrac{n-l}{2} \right)!}
\end{align*}
\end{proposition}

\begin{proof}
It follows from Proposition \ref{prop.dimCYgraph} that
\begin{align*}
\dim A_n &= \sum_{(\lambda,\mu) \in \Theta_n} \dim_{\Theta} \left((\lambda,\mu)\right)^2 \\
&= \sum_{|\mu| = n \mod 2 }  \left(\binom{n}{|\mu|} (n-|\mu|-1)!! \dim_{\Y}(\mu) \right)^2 \sum_{|\lambda| = \tfrac{n-|\lambda|}{2}}\left(\dim_{\Y}( \lambda)\right)^2 \\
&= \sum_{|\mu| = n \mod 2 }  \left(\binom{n}{|\mu|} (n-|\mu|-1)!! \dim_{\Y}(\mu) \right)^2 \left(\frac{n-|\lambda|}{2} \right)!
\end{align*}
If $n=2m+1$ is oneven, this is equal to
\begin{align*}
&\sum_{k=0}^m \left(\sum_{|\mu|=2k+1} \dim_{\Y}(\mu)^2 \right) \left(\binom{2m+1}{2k+1}(m-k)! (2(m-k-1))!! \right)^2 \\
&= \sum_{k=0}^m (2k+1)! \left(\binom{2m+1}{2k+1}(m-k)! (2(m-k-1))!! \right)^2 \\
&= \sum_{k=0}^m \frac{((2m+1)!)^2}{(2k+1)! (m-k)! 2^{2(m-k)}}.
\end{align*}
The $n=2m$ case follows from an analogous computation.
\end{proof}

In \cite[Section 3.3]{Or05}, Orellana describes a presentation of the centralizer algebra $B_n(N)$ of the hyperoctahedral group $H_N$ in the $n$-th tensor power of its standard representation. The presentation is given in terms of generators $b_i,e_i,s_i, \ 1 \leq i \leq n-1$ where the $s_i$ generate a copy of the symmetric group algebra $\C[S_n]$ and where the $e_i/ \sqrt{n}$ are Jones projections. We will now define algebras $A_n$ such that $A_n$ is the quotient of $B_n(N)$ by the ideal generated by the Jones projections $e_i, \ 1\leq i \leq n-1$.

\begin{definition}
The \emph{coupled Young algebra} $A_n$ or order $n$ is the unital $*$-algebra generated by selfadjoint elements $b_i,s_i, \ 1 \leq i \leq n-1$ and relations
\begin{align*}
s_i^2 = 1, \qquad b_i^2 &= b_i \qquad s_i b_i = b_i s_i = b_i, \qquad &\text{for } 1 \leq i \leq n-1, \\
s_i s_j = s_j s_i, \qquad b_i b_j &= b_j b_i, \qquad s_ib_j = b_j s_i,  \qquad &\text{  for }  |i-j| \geq 2, \\
s_i s_{i+1} s_i = s_{i+1} s_i s_{i+1} \qquad b_i b_{i+1} &= b_{i+1}b_i = 0 \qquad s_i s_{i+1} b_i s_{i+1}s_i = b_{i+1}, \qquad &\text{for } 1 \leq i \leq n-2.
\end{align*}
\end{definition}

\begin{oproblem}
Determine the representations of the algebras $A_n$ and prove in particular that the branching graph of the sequence $A_0 \subset A_1 \subset \dots$ is the coupled Young graph.
\end{oproblem}

\end{document}